\documentclass{amsart}

\usepackage{booktabs,xcolor}

\newtheorem{theorem}{Theorem}[section]
\newtheorem{lemma}[theorem]{Lemma}

\theoremstyle{definition}

\newtheorem{example}[theorem]{Example}

\newtheorem{remark}[theorem]{Remark}
\newtheorem{hypothesis}[theorem]{Assumption}

\theoremstyle{remark}

\numberwithin{equation}{section}

\newcommand{\al}{\alpha}

\newcommand{\de}{\delta}

\newcommand{\auskommentieren}[1]{}

\newcommand{\beq}{\begin{equation}}
\newcommand{\eeq}{\end{equation}}
\newcommand{\bea}{\begin{equation}\begin{aligned}}
\newcommand{\eea}{\end{aligned}\end{equation}}

\newcommand{\Mcc}{\mathcal{M}}
\renewcommand{\deg}{deg}

\newcommand{\red}[1]{\textcolor{black}{#1}}

\newcommand{\dd}{{\rm d}}

\newcommand{\norm}[2]{\left\lVert#1\right\rVert_{#2}}

\usepackage[colorlinks,citecolor=red,urlcolor=blue,bookmarks=false,hypertexnames=true]{hyperref}

\DeclareMathOperator{\dist}{dist}
\DeclareMathOperator{\dive}{div}

\DeclareMathOperator{\diam}{\operatorname{diam} }
\DeclareMathOperator{\id}{id}

\usepackage{mathtools}
\DeclarePairedDelimiterX\set[1]\lbrace\rbrace{\def\given{\;\delimsize\vert\;}#1}

\usepackage[utf8]{inputenc}

\begin{document}

\title[Error estimates for finite elements for geometric equations]{\red{A priori error estimates for finite element approximations of regularized level set flows in higher norms}}

\author{A. Kr\"oner}
\curraddr{}
\address{Weierstrass Institute for Applied Analysis and Stochastics (WIAS), Mohrenstr. 39, 10117 Berlin, Germany}
\email{axel.kroener@wias-berlin.de}

\author{H. Kr\"oner}
\address{Universit\"at Duisburg-Essen, Fakult\"at f\"ur  Mathematik, Thea-Leymann-Stra\ss e 9, 45127 Essen, Germany}
\email{heiko.kroener@uni-due.de}


\date{January 1, 2001 and, in revised form, June 22, 2001.}



\begin{abstract} 
\if{In the context of geometric evolution equations a common method is to use
a level set function to model the evolution of surfaces in the three dimensional Euclidean space.  While $C^0$ is a natural 
regularity in the context of viscosity solutions in special cases higher regularity is available and even necessary in order to detect level sets from the level set function in the geometrically non-pathological case of smooth surfaces as level sets.  Correspondingly, for a finite element scheme approximating the level set equation one would like to also have error estimates which usually come naturally in Sobolev norms at least in $W^{2,\mu}$, $\mu$ appropriately, which then imply via embedding theorems 
an error estimate in $C^1$ and hence give meaningful information about the approximation error of the level set surface itself.}\fi
\red{This paper proves error estimates for $H^2$ conforming finite elements for equations which model the flow of surfaces by different powers of the mean curvature (this includes mean curvature flow).
The scheme is based on a known regularization procedure and produces different kinds of errors, a regularization error, a finite element discretization error for the regularized problems and a full error. While in the literature and own previous work different aspects of the aforementioned error types are treated, here, we solely and for the first time focus on the finite element discretization error in the $W^{2,\mu}$ norm for the regularized equation analyzing also the dependencies from the regularization parameter.} \end{abstract}

\maketitle

\tableofcontents

\section{Introduction}\label{sec::main_result}
\if{The purpose of our paper is abstractly said to shed light on the theoretical background of a numerical scheme which is already used in the literature. Since our specific type of problem has not received that much attention before we keep the structure of the introduction of rather longer and expository  type in order to explain the problem, (own) previous work and to motivate our solution.
}\fi
 \if{This paper extends specific aspects of the  finite element error analysis for regularized geometric evolution equations considered in the work  \cite{HK} of the second author. Before we go into the details of our problem we give a very short introduction to geometric  evolution equations for the non-specialist in this field.}\fi
  \red{This paper treats a specific problem concerning the numerical analysis of a geometric evolution equation.   Geometric evolution equations
  are in a certain sense a geometric analog of the heat equation but appear already in their simplest shape  as a nonlinear system of PDE. In order to illustrate this we let $x:I \times \Omega\rightarrow \mathbb{R}^3$, $\Omega \times  I\subset \mathbb{R}^2\times \mathbb{R}$ open, $I$ interval, $x=x(t, \xi)$, be a smooth family of immersions $x(t, \cdot)$. 
 We consider the homogeneous heat equation with  
 the Laplace Beltrami operator $\Delta_{\xi}$ instead of the usual Laplacian and the unknown function replaced by $x$, yielding
 \begin{align}
 \frac{\dd}{\dd t}x(t, \xi) = \Delta_{\xi} x(t, \xi)
 \end{align}
 which is the defining equation of mean curvature flow  in the parametric formulation.
Herein, we may write  $\Delta_{\xi} x = -H \nu$ with $H$ the mean curvature  and $\nu$ the outward directed normal vector of the surface parameterized by $x(t, \cdot)$, according to the Gauss formula from differential geometry. We remark that mean curvature flow can also be interpreted as the gradient flow for the area functional. Under mean curvature flow a smooth, convex initial hypersurface shrinks in finite time to a point and that singularity looks round `under the microscope', i.e. after rescaling,  see \cite{MR1216584}. 
There are variants of mean curvature flow in the literature at which $-H\nu$ is replaced in the right-hand side above by $(1/H)\nu$ (inverse mean curvature flow) or also $-H^k\nu$, $k\ge 1$, (flow by a power of the mean curvature). We refer here to the introductions of \cite{MR1916951} and \cite{MR2420018} for a good overview over these variants and will discuss aspects and applications of these flows more detailed in the sequel.
}

\red{
The parametric formulation has limitations as one can already see from the example of mean curvature flow. When the initial surface is not convex, e.g. a shape of a dumbbell, it might happen that the bar of the dumbbell vanishes somewhere so that the surface breaks into two (non-empty) pieces. Such phenomena can be seen in the parametric formulation only before such topological changes occur which themselves are known only from suitably generalized formulations.
Examples are the level set formulation \cite{MR1100206, MR1374010,MR1100211}, a framework to describe mean curvature flow in the sense of geometric measure theory \cite{Brakke}, and a formulation as a limit of phase field models \cite{AllenCahn:1979}. }

In the present paper we only study a level set formulation which in turn allows topological changes of the evolving level set.
Thereby we extend specific aspects of the  finite element error analysis for regularized geometric evolution equations in level set form considered in   \cite{HK} by the second author. 
In  \cite{HK}  error analysis for a finite element approximation  for the following family of  equations (which is here slightly more general than in that reference) is considered:
\begin{equation} \label{unified_equation}
\begin{cases}
\begin{aligned}
\dive \left(\frac{\nabla u^{\varepsilon}}
{\sqrt{|\nabla u^{\varepsilon}|^2+\varepsilon^2}}\right)-
\eta \left(\sqrt{ |\nabla u^{\varepsilon}|^2+\varepsilon^2 }\right)&=  0 && \text{in } \Omega_{\alpha}, \\
u^{\varepsilon} &= f_{\alpha}  && \text{on } \partial \Omega_{\alpha}
\end{aligned}
\end{cases}
\end{equation}
with 
\begin{equation}
\eta(r):= \sigma r^{\alpha}.
\end{equation}
This family of equations was introduced in \cite{MR1916951} and \cite{MR2420018} in order to approximate weak solutions of inverse mean curvature flow and the flow by (powers of) mean curvature. The appearing quantities will be explained in the next section. 

\subsection{More details about the Equation \eqref{unified_equation}} 
Let $\alpha \in \mathbb{R}$ and
assume that $\Omega_{\alpha} \subset \mathbb{R}^3$ is an open, bounded domain with a smooth boundary
$\partial \Omega_{\alpha}$.

If $\alpha>0$ we assume that we can rewrite $\partial \Omega_{\alpha}=\partial A \cup \partial B_R(0)$. Here $A$ is an 
open, bounded set containing the origin with $\partial A$ having positive mean curvature and playing the role of the initial hypersurface. Note that the sign convention is so that the mean curvature is taken with respect to the outer normal vector so that a ball has positive mean curvature (with respect to its outer normal vector). Hence $\partial A$ as subset of $\partial \Omega_{\alpha}$ has negative mean curvature. $B_R(0)$ is a large ball of radius $R>0$ around the origin 
  so that $R$ is large compared to $\diam A$. 
Let $f_{\alpha}=0$ on $\partial A$ and $f_{\alpha}=L$ on $\partial B_{R}(0)$ where $L$ is a large constant. In case $\varepsilon=0$ we have the following: Around a point $p \in \Omega_{\alpha}$ at which the level set of $u^{\varepsilon}$  (here it even suffices that $u^{\varepsilon}$ solves the equation without prescribing boundary values) is  $C^2$ with non-vanishing gradient the level set evolves in a neighborhood of  $p$ in  the direction of its outer normal with speed given by $1/H^{1/\alpha}$. The evolving surface at time $t$ is defined here as the $t$ level set of the level set function. This flow is an expanding flow and moves outward.
 
 If $\alpha<0$ we let $\partial \Omega_{\alpha}$ have positive mean curvature and  $\partial \Omega_{\alpha}$ will play the role of the initial hypersurface of the curvature flow. We assume that $f_{\alpha}=0$. Now assuming $\varepsilon=0$
 and the same regularity of a level set around a point $p \in \Omega_{\alpha}$ as above the level set evolves in the direction of the inward unit normal vector with speed equal to $H^{1/\alpha}$.
 
 As far as these models appear in the literature (see \cite{MR1916951}, \cite{MR2420018} and references therein) the exponents in the speed and the exponents in the equation are related as follows: For the case $(\sigma,\alpha)=(1,1)$ we have a flow along the outer normal vector with speed given by the (multiplicative) inverse of the mean curvature (so-called inverse mean curvature flow) and in the case $(\sigma, \alpha)=(-1,-1/k)$ a contracting flow by a power $k$, $k\ge 1$, of the mean curvature (this includes mean curvature flow for $k=1$). 
 
The parameter $\varepsilon$ is a nonnegative regularization parameter which prevents the  equations 
from becoming singular if it is positive. The idea in the literature was to establish an existence theory for the above equations in the case of non-vanishing regularization parameter $\varepsilon>0$ and then taking the limit to solve the  equation of real interest. Taking that limit involves in the case $\sigma=\alpha=1$
the triple of parameters $(\varepsilon, R, L)$ and in the case $(\sigma,\alpha)=(-1,-1/k)$, solely the parameter $\varepsilon$. In the first case taking the limit is even a more challenging process which involves a certain relation between $\varepsilon$, $R$ and $L$ and leads to a weakly defined inverse mean curvature flow, see  \cite{MR1916951}. 
While basically taking into account all parameters is desirable our current methods are only able to perform finite element error analysis with qualitative constants with respect to the parameter $\varepsilon$. 
Therefore we make the following assumption throughout the paper. 
\begin{hypothesis}\label{hyp_new}
For simplicity we will restrict the discussion in our paper to the case $(\sigma,\alpha)=(-1,-1/k)$, $k\ge 1$, write $\Omega=\Omega_{\alpha}$ and set $f_{\alpha}=0$. 
\end{hypothesis}
This simplification essentially inherits already all difficulties arising solely from the parameter $\varepsilon$ which we want to discuss and present in the paper. 
 It simplifies
the situation of inverse mean curvature flow by omitting $L$ and $R$ as parameters as well as by simplifying the domain.
 Note that the consideration of a fixed 3D ring-type domain with fixed non-zero boundaries at the outer boundary without taking care of qualitative constants in finite element error estimates for the inverse mean curvature flow can be found already in the literature in \cite{MR2350186}. 

We recall some properties of the equations  in the case of a non-vanishing regularization parameter:  equation \eqref{unified_equation} has smooth solutions $u^{\varepsilon}$ for each $0<\varepsilon<\varepsilon_0$, 
$\varepsilon_0>0$ sufficiently small, cf. \cite{MR2420018}. 
The regularized solutions $u^{\varepsilon}$ converge locally uniformly in $\Omega$ to a function $u$ as $\varepsilon \rightarrow 0$.
Here, $u$ is a  continuous viscosity solution of \eqref{unified_equation} for $\varepsilon=0$, cf. \cite{MR2420018}. 
In order to solve equation  \eqref{unified_equation} computationally with $\varepsilon=0$, it is tempting and reasonable to circumvent the possible singularity of the equation by solving \eqref{unified_equation} for $\varepsilon>0$ small and fixed instead.  In this paper we also follow this approach. 

\red{Concerning applications and (an incomplete list of) major contributions to these flows} we mention the following facts about mean curvature flow and inverse mean curvature flow. Mean curvature flow was mathematically introduced in \cite{Brakke}
and serves \red{as a description} of the physical system given by the motion of grain boundaries in an annealing piece of metal  \cite{Mullins}.  
Parametric, smooth inverse mean curvature was discovered as crucial tool to prove the Riemannian Penrose Inequality 
in a special case, see \cite{MR523907} and references therein. 
By using the concept of weakly defined inverse mean curvature the above mentioned restriction to the special case could be widely removed  and a proof of that inequality was achieved in \cite{MR1916951} under the less restrictive assumption of a so-called connected apparent horizon.  In a further step the connected apparent  horizon assumption could be removed in \cite{MR1908823} with  a completely different geometric flow. The formulation of the Riemannian Penrose Inequality goes back to Penrose \cite{penrose} and has relevance in general relativity.

\subsection{The numerical analysis perspective on \eqref{unified_equation} and our contribution}\label{sec:contr} 
We focus on the numerical approximation of  \eqref{unified_equation} for $\varepsilon>0$. We derive error estimates where we show an explicit dependency on $h$ and a qualitative dependency on $\varepsilon$. 
\red{The work \cite{MR2350186} shows estimates between the regularized solution $u^{\varepsilon}$ of \eqref{unified_equation} and its corresponding discrete solution $u^{\varepsilon}_h$ of type}
\red{\begin{equation} \label{standard_case}
\begin{aligned}
 \norm{u^{\varepsilon} - u_h^{\varepsilon}}{W^{1,p}}\le c(\varepsilon) h,\quad \norm{u_{\varepsilon} - {u}^{\varepsilon}_h}{L^{p}}\le c(\varepsilon) h^2
\end{aligned}
\end{equation}}
for a finite element approximation of  \eqref{unified_equation} for some $\varepsilon>0$ fixed in the case $ (\sigma,\alpha)=(1,1)$  (inverse mean curvature flow) and discretization parameter $h>0$. 
The constant $ c(\varepsilon) $ depends in an unspecified way on $\varepsilon$. Note that this paper considers as domain a ring-type region and boundary values zero at the inner boundary and $L>0$ at the outer boundary, compare also our Assumption \ref{hyp_new}.
In \cite{HK} we were able to derive an estimate of these constants in form of an upper bound 
\red{\begin{align}
c(\varepsilon)\le C e^{P(1/\varepsilon)},
\end{align}}
with $P$ a polynomial in one variable, for the price that the approximation is based on piecewise cubic and globally $H^1$ conforming finite elements, the space dimension is three and that the resulting powers of $h$ on the right-hand side of \eqref{standard_case} depend on certain parameters but in an explicit way.
The contribution of this paper is to extend the results from the latter reference to an $H^2$--conforming approach and to state sufficient conditions for the parameters $\varepsilon$ and $h$ in order to get quantitative error bounds,
i.e. for $\mu\ge 2$ we derive estimates of type
\red{\begin{align}
  \norm{u^{\varepsilon} - u_h^{\varepsilon}}{W^{2,\mu}}\le c e^{P(1/\varepsilon)}h^{\delta}
\end{align}}
for a suitable positive  $\delta$.

Agreeing to the fact that finite element error estimates usually come in Sobolev norms (and excluding here in the sense of a first and not too technical  approach Sobolev spaces of non-integer differentiability  orders)
$H^2$--conformity is the lowest degree of conformity which allows
to obtain via embedding theorems error estimates in the $C^1$ norm, cf. Section~\ref{sec:3}. 
Note that $H^2=W^{2,2}$ does not embed into $C^1$ in general, but that for the considered finite element space $H^2$ conformity is a useful terminus, since it implies $W^{2,p}$ conformity for any $p\ge 2$ by definition of the discrete space. $H^2$ conformity in three dimensions is practically rather elaborate while it is a well-established approach at least in the two-dimensional case; cf. the Argyris element or the Clough-Tocher element for the biharmonic 
equation, \red{see also \cite{MR1930132}. 
For further applications of the Clough–Tocher (and also Hseigh--Clough--Tocher element) beyond biharmonic problems, we refer to 
\cite{DominicusGaspozKreuzer,MR4199789,MR2755946,KaweckiSmears:2020a,KaweckiSmears:2020b,MR3921147}. }
We remark that also nonconforming approaches have shown their relevance in numerical computations. Here, in this context we do not follow their approaches since we want  to get $C^1$ error estimates at the end but from a purely practically point of view they could be interesting for further studies.

 In the framework of viscosity solutions $C^0$ is a natural norm  and hence also for numerical error estimates of their discrete approximations. In the special application of the modeling of evolving level sets  the $C^1$-norm is the most natural norm which gives insight about the precise location of level sets. 
Around a point where the level set function $u$ is in $C^1$ with 
\begin{equation} \label{gradient_nonzero}
|Du| \ge \delta_0,
\end{equation} 
$\delta_0>0$ a constant, the error estimate translates into error estimates for the level sets via the implicit function theorem. Furthermore, it
allows to find regions where a condition of type \eqref{gradient_nonzero} holds at all.

\if{
\bigskip
OLD: the crucial improvement of the upper bound $e^{P(1/\varepsilon)}$ derived in \cite{HK} to 
$P(1/ \varepsilon)$ here, i.e. the newly derived expression depends only polynomially instead of exponentially on $\frac{1}{\varepsilon}$. We will show such a result under the additional assumption of $H^2$--conformity of the considered finite elements.
}\fi

 
\subsection{Motivation of the paper and relation to existing literature} 
The motivation to work out explicit dependencies of the constants on the regularization parameter is to obtain {\it full} error estimates. The latter type of estimates measures
the error between the solution of the discrete approximation and the solution of the equation \eqref{unified_equation} with $\varepsilon=0$. It is temptative for it to go via the triangle inequality over
 the solution of the regularized equation. In doing so explicit constants are very useful. 
The quantitative convergence results for approximations of geometric evolution equations 
 developed in this paper continue our previous work
\cite{MR3892432, MR3734408, HK} and focus for the first time on the higher order conforming case.
These references and the present paper try to identify the simplest setting given the scheme in \cite{MR2350186}  in which full error estimates are available which give information about the level sets.
Furthermore, this paper is
inspired from \cite[Rem.~4]{MR2350186}, \cite{MR1417861}, and \cite{MR1760409}. 

In \cite[Rem.~4]{MR2350186},  a question is raised asking about more explicit dependencies of the constants on the regularization parameter $\varepsilon$ than presented in that reference. Note that concerning that question our paper studies a crucially simplified setting, see \ref{sec:contr}. On the other hand we think that in order to answer that question it is a useful step to understand first this simplified setting.
In \cite{MR1417861} the authors show the convergence of a difference scheme to parabolic level set mean curvature flow and \cite{MR1760409} proves bounds for constants in error estimates (and in the sequel full error estimates)  for a scheme motivated from \cite{MR1417861} which are polynomial in the inverse regularization parameter. 
We think that  
it is desirable to have something similar (meaning convergence with dependencies on the parameters) for an elliptic problem and finite elements.

The authors are not aware 
of full error estimates in the literature in (at least) $C^1$--norms  for elliptic level set approaches to mean curvature flow or even for the parabolic level set mean curvature flow. 
\if{Apart from the fact that this is clearly an improvement the second author became aware of the relevance of improving available exponential rates in numerical error estimates to polynomial rates by \cite{MR1971212} where numerical analysis  for the Allen-Cahn equation is considered.
There, polynomial rates are derived using spectral estimates from \cite{MR1284813}. 
}\fi
 We refer also to the very recent work \cite{lubich}  about convergence and error estimates 
 in the $H^1$ norm for a finite element scheme for parametric mean curvature flow of closed surfaces which can be understood  as full error estimates with respect to their approach. 
For additional references about numerical analysis for geometric evolution equations in general we refer to the variety of references in \cite{MR3892432, lubich,MR2350186,MR2168343} as well as to these papers themselves.
Furthermore, we mention concerning rates of convergence with respect to a regularization parameter also a fourth order regularization of the radially symmetric Monge-Amp\`ere equation in \cite{MR3219673}. Here polynomial rates in the regularization parameter are shown under certain assumptions on the solution (and solution of the regularized problem).

\if{\textcolor{blue}{In  the current paper we give a strong theoretical indication that in general one cannot expect a better dependence on the regularization parameter than exponential. 
%
REMOVE: shown here by using $H^2$ conforming elements and as consequence controlling the error  in $W^{2,p}$ norms which is convenient for the underlying motivation of tracking level sets.
The study of the lower conforming case
will be considered in future work.}}\fi


\if{
 There are several $H^2_0$--conform elements in two dimensions, as the Argyris element with $P_A:=(P(K),\dim P(K))=21$, the Clough-Tocher element $P_C:=
(\{\varphi \in C^1(K)|\varphi|_{K_i} \in \in  P_3(K_i), i = 1,2,3 \},12)$, and the square element $P_S:=(Q_3(K),16)$ which are well established methods e.g. for the XXX equation or the XXX equation; however the condition of the corresponding matrices is 
very bad which motivates to consider mixed formulations.
}\fi


\if{There are several $H^2_0$--conform elements in two dimensions, as the Argyris element with $P_A:=(P(K),\dim P(K))=21$, the Clough-Tocher element $P_C:=
(\{\varphi \in C^1(K)|\varphi|_{K_i} \in \in  P_3(K_i), i = 1,2,3 \},12)$, and the square element $P_S:=(Q_3(K),16)$ which are well established methods e.g. for the XXX equation or the XXX equation; however the condition of the corresponding matrices is 
very bad which motivates to consider mixed formulations. (fuer welche Gleichung? fuer welche Matrix?). In 3D XXX
}\fi

 \if{We prove an upper bound for constants in standard finite element error estimates
 given by a polynomial expression of the inverse powers of the regularization parameter. In \cite{HK} the dependence was exponential expression in inverse powers of the regularization parameter as in 
 where we derived such a relation for a scheme taken from \cite{MR2350186} for equation \eqref{unified_equation}, see below.
 This improvement is currently achieved at the price of stronger assumptions on the degree of conformity of the finite elements, i.e. we now need $H^2$ instead of solely $H^1$--conformity.}\fi

\subsection{Notation} Throughout the paper we use for a bounded domain $\Omega \subset \mathbb{R}^3$ the usual notation for Lebesgue spaces $L^p(\Omega)$, $1\le p \le \infty$, and Sobolev spaces $W^{m,p}(\Omega)$, $0\le m < \infty$, $1\le p\le \infty$. For $p=2$ we write $H^m(\Omega):=W^{m,2}(\Omega)$. By $H^1_0(\Omega)$ we denote the closure of the smooth functions being compactly supported in $\Omega$ in the $H^1$--norm; its dual space is denoted by $H^{-1}(\Omega)$.  We denote generic constants in estimates usually by $c$ and use the summation convention that we sum over repeated indices. By $|\cdot|$ we denote the Euclidean norm, and by $|\cdot|_1$ the $1$-norm. We write $z=(z_1,z_2,z_3)^{\top} \in \mathbb{R}^3$. The Landau symbols are denoted by $\mathcal{O}(\cdot)$ and~$o(\cdot)$.

Note that the method presented here generalizes to several other situations and equations which inherit certain parameters for which a certain asymptotic is assumed.

{

\section{Main result}\label{sec:main_result}
We adopt the setting from  \cite{HK}.
 The domain of the level set function is considered to be three dimensional, so that level sets 
are in case of sufficient regularity two dimensional surfaces.
Let 
\begin{align}
\{\mathcal{T}_h: 0<h<h_0\}
\end{align}
be a family  of shape-regular and \red{quasi-uniform} triangulations (of tetrahedra) of $\Omega$, $h$ the mesh size of  $\mathcal{T}_h$ and $h_0=h_0(\Omega)>0$ small,
where we allow boundary elements to be curved and define
\beq \label{ref_link}
\Omega^h := \bigcup_{T\in \mathcal{T}_h}T;
\eeq
we have not in general $\Omega^h\subset \bar \Omega$.
We will specify the triangulation concerning its boundary approximation in the following two assumptions
and will already here stipulate that we will consider piecewise polynomial functions with polynomial degree $\operatorname{deg}$ at most $\operatorname{deg}\ge 8$, $\operatorname{deg} \in \mathbb{N}$, see Remark \ref{rem:101}. Note that the
last inequality will turn out to be necessary for technical reasons (otherwise \red{we are not able to obtain} a nonempty parameter range for $\delta$ in Theorem \ref{main_thm_endversion}) and that the `at most' condition arises from 
the widely used convention to work with functions which are piecewise polynomials and globally continuous of a prescribed maximal degree which we adopt here as well. 

\begin{hypothesis}\label{hyp1} For $h>0$ we assume the existence of the following $H^2$-conforming finite element space:
\beq
\begin{aligned} \label{def_V_h_}
V_h:=& \left\{
\begin{array}{l} w \in C^1(\bar \Omega^h) \; \bigg| 
\begin{array}{l}\text{for all } T\in \mathcal{T}_h:\;\; w|_T \text{ polynomial of degree} \le \deg \\
 \text{(up to the  transformation in case $T$ is a}\\
  \text{boundary cell)},\quad 
  w_{|\partial \Omega^h}=0.
  \end{array}
\end{array}
\right\}.
\end{aligned}
\eeq
Hereby, we allow for the finite elements restricted to the boundary cells to be accordingly transformed polynomials, see Appendix \ref{sec:real}.
\end{hypothesis}
The corresponding space when dropping the requirement $w_{|\partial \Omega^h}=0$ in \eqref{def_V_h_} is denoted by $\tilde V_h$.

\red{Let $d: \mathbb{R}^{3}\rightarrow \mathbb{R}$  be the signed distance function of $\partial \Omega$ where the sign convention is so that $d$ is negative inside $\Omega$ and nonnegative outside $\Omega$.
Let $\de_0 = \de_0(\Omega)>0$ be small and define
for $0<\de<\de_0$ that
\beq \label{notation_dist}
\Omega_{\de}:= \{x\in \mathbb{R}^3\;|\; d(x)<\de\}.
\eeq}
\begin{hypothesis}\label{hyp2}
For  $0<h\le h_0$ we assume that there exists an interpolation operator
 $\mathcal{I}_h:W^{\deg,\infty}(\Omega^h)\rightarrow \tilde V_h$
such that for $1\le p \le \infty$ 
\beq \label{866}
\|u-\mathcal{I}_hu\|_{W^{m,p}(\Omega^h)} \le ch^{\deg-m}\|u\|_{W^{{\deg,p}}( \Omega^h)} \quad \forall \; u \in W^{\deg,\infty}( \Omega^h), \quad m =1, 2.
\eeq
\if{\beq \label{866}
\|u-\mathcal{I}_hu\|_{W^{m,\infty}(\Omega^h)} \le ch\|u\|_{H^{{m+1}}( \Omega^h)} \quad \forall \; u \in H^{\deg}( \Omega ^h)
\eeq
\beq \label{866}
\|u-\mathcal{I}_hu\|_{H^{m}(\Omega^h)} \le ch\|u\|_{H^{{m+1}}( \Omega^h)} \quad \forall \; u \in H^{\deg}( \Omega ^h)
\eeq
\beq \label{866}
\|u-\mathcal{I}_hu\|_{H^{2}(\Omega^h)} \le ch^8\|u\|_{H^{{10}}( \Omega^h)} \quad \forall \; u \in H^{10}( \Omega ^h).
\eeq}\fi
Furthermore, 
there is a constant $0<\tilde c:=\tilde c(\Omega)$ so that
\beq \label{851}
\partial \Omega^h \subset \Omega_{\tilde c h^{\widetilde \deg} }\backslash \Omega_{-\tilde c h^{\widetilde \deg}}. 
\eeq
where we use the notation
\eqref{notation_dist} and assume that 
\begin{equation} \label{inequ}
  {\widetilde \deg}\le \deg/2, \quad {\widetilde \deg} \in \mathbb{N}.
\end{equation} 
\end{hypothesis}

Note that the inequality \eqref{inequ} 
is assumed for simplification and is to avoid a later consideration of two different cases and that the excluded range for $\widetilde \deg$ could be treated analogously but is not that interesting since it corresponds to a rather high boundary approximation. 
\red{Moreover, note that in the main Theorem \ref{main_thm_endversion} we will assume that $\widetilde{\deg} \ge 4$ which implies by \eqref{inequ} $deg \ge 8$, see also Example  \ref{example}.}
\if{While the hard reason for large values of 
$  {\widetilde \deg}$ arises from a necessary lower bound for $  {\widetilde \deg}$, see Remark \ref{rem:101}, the realizing Example \ref{example} uses a  value for $\deg$ which is even larger than required by \eqref{inequ}.}\fi

Throughout the paper we assume Assumptions  \ref{hyp_new}, \ref{hyp1} and \ref{hyp2}.


We denote the set of \red{vertices} of $\mathcal{T}_h$ by $N_h$. 
We recall that a continuous piecewise polynomial 
function whose derivatives are continuous is an element in $H^2(\Omega)$, cf. \cite[Thm. 5.2]{MR2322235}.
Since the curved elements at the boundary can be treated in the weak formulation of the discrete problem analogously as if they were exact tetrahedra \red{in the following we will refer to them  as to the usual tetrahedra.}
\red{We have }
\bea
\partial \Omega_{\de}\in C^{\infty},\quad \|\partial \Omega_{\de}\|_{C^{2}}\le c(\Omega)\|\partial \Omega\|_{C^{2}}.
\eea
Let $h_0>0$ be chosen such that 
\bea\label{h_0}
\overline{\Omega^{h}}\subset \Omega_{\delta_0} \quad\text{ for all } 0<h\le h_0.
\eea

We extend $u^{\varepsilon}$ (which is a $C^{\infty}$ smooth solution of \eqref{unified_equation}) to a function in $C^m(\Omega_{\de_0})$ with $m\in \mathbb{N}$ 
sufficiently large, denote the extension again by $u^{\varepsilon}$, and assume that 
\beq \label{865}
\|u^{\varepsilon}\|_{C^m(\Omega_{\de_0})}\le c\|u^{\varepsilon}\|_{C^m(\bar\Omega)}.
\eeq   
This extension is achieved by first performing a localization via a smooth partition of unity and then  a local $C^m$-reflection at the boundary. The whole procedure is standard in classical theory for partial differential equation, see \cite[Chapt. 7]{MR737190}. 




\if{For convenience we recall the following inverse estimate which is not formulated in a most general form but suitable for our purposes and which will be used several times in the paper without mentioning it every time, see \cite[Sec. 4.5]{BS} for a proof.
\begin{lemma} \label{inverse_estimate_lemma}
For $1\le p,q\le \infty$ there exists 
a constant $c>0$ such that
\begin{equation}
\|v_h\|_{W^{1,p}(\Omega^h)} \le c h^{\frac{3}{p}-\frac{3}{q}}\|v_h\|_{W^{1,q}(\Omega^ h)}
\end{equation}
for all $v_h \in V_h$.
\end{lemma}
}\fi

Furthermore, when (tacitly) extending functions $v_h \in V_h$ to $\mathbb{R}^3$ by zero we denote the extended function again by $v_h$. 

We formulate our main result.

\begin{theorem} \label{main_thm_endversion} Let Assumption \ref{hyp_new}, \ref{hyp1}, and \ref{hyp2} hold. 
Let $\mu\ge 2$, $\mu\in \mathbb{R}$, and $\widetilde{\deg} \ge 4$, $\widetilde{\deg}\in \mathbb{N}$, be given.
\red{Setting  $\nu:={\widetilde \deg} -2 +\frac{\widetilde \deg}{\mu}$ we choose  $\delta$ such that 
\begin{align}\label{cond:parameter1}
\left\{\begin{array}{ll}
 \min\left(\nu,\frac{3}{\mu} -\frac{7}{2} + \frac{4}{3} \widetilde{\deg} \right) > \delta > \frac{3}{2},&\text{if }\mu>3,\\
 \nu > \delta > \frac{3}{\mu} + \frac{1}{2},& \text{if }2\le \mu\le 3.
 \end{array}
 \right.
\end{align}}
%
Then for
\beq \label{tolerance}
\rho:=e^{\varepsilon^{-\gamma}}h^{\de}
\eeq
and  $\gamma>0$ a suitable constant (which can be calculated explicitly) 
the following holds: 
For every $0<\varepsilon<\varepsilon_0$ and $0<h \le h_0$ the equation
\beq \label{930}
\int_{\Omega^h} \frac{\left<D u^{\varepsilon}_h, D \varphi_h\right>}{\sqrt{\varepsilon^2+|D u^{\varepsilon}_h|^2}} = -\int_{\Omega^h}\eta\left(\sqrt{\varepsilon^2+|D u^{\varepsilon}_h|^2}\right)
\varphi_h \quad \forall \; {\varphi_h\in V_h },
\eeq
has a unique solution $u^{\varepsilon}_h$ in 
\beq \label{912}
\bar B^h_{\rho}:=\{w_h\in V_h: \|w_h-u^{\varepsilon}\|_{W^{2, \mu}(\Omega^h)} \le \rho \}.
\eeq
\end{theorem}
\begin{remark}\label{rem:101}
We have $4\le \widetilde{\deg}\le \frac{\deg}{2}$, see the discussion below Assumption~\ref{hyp2}, and hence $\deg\ge 8$.

 Note that the assumption $\mu \ge 2$ is chosen here since it corresponds to a natural range for the orders of the Sobolev spaces in the error estimate and will be assumed throughout the paper. Furthermore, it includes the desired case $\mu>3$ such that we have the continuous embedding of $W^{2,\mu}$ into $C^1$. 
 
 The theorem does not provide information about existence/uniqueness of a solution of the discrete problem outside the specified neighborhood. 
 
 Our perspective and the purpose of the theorem is to study  {\it theoretical} properties for the numerical scheme from \cite{MR2350186} and hereby to supplement findings therein for a fixed regularization parameter with asymptotical analytical aspects arising from a variable regularization parameter (in combination with higher order Sobolev norms).
  \end{remark} 

\begin{example}\label{example}
 \red{We give  examples for possible choices.
Set $\widetilde{\deg}=4$ and $\deg=9$ then sufficient conditions for $\delta$ read as follows:
\begin{align}\label{cond:parameter1b}
2<\delta<4                    \quad  \text{ for } \mu=2,\quad\quad \quad 
3/2<\delta<4 \quad \text{ for } \mu=4.
 \end{align}
For a realization see Section \ref{sec:real}.
 The reader should have in mind this choice for $\widetilde{\deg}$ and $\deg$ as the prototype example and the simplest scenario known to the authors.
 }
 \end{example}
\if{
 (i) For $\deg=\widetilde{\deg}=9$ (see for a realization Section \ref{sec:real}) the condition for $\delta$ reads as follows:

\begin{align}\label{cond:parameter1b}
 
\end{align}
%
and for $2\le \mu\le 3$ 
\begin{align}\label{cond:parameter2b}
 \frac{3}{\mu}  + \frac{17}{2} \ge \delta \ge \frac{3}{\mu} - \frac{1}{2}.
\end{align}
}\fi
%
%
%
In comparison with our previous result \cite[Thm. 1.2]{HK} there are 
three main differences:
\begin{itemize}
\item[(i)]  In \cite{HK}  we used $H^1$ instead of $H^2$--conforming finite elements. 
Correspondingly,  in \cite{HK}  we estimated $W^{1, \mu}$- instead of $W^{2, \mu}$--norms as in the present paper.

\item[(ii)]  The dependencies on $h$ and $\varepsilon$ specifying $\rho$ for the $H^2$ case are derived here. 
The derivation requires extra effort compared to the standard one which does not lead to such expressions of $\rho$ in terms of $\varepsilon$ and $h$. Here the `standard one' is meant to be the simplified situation when in \eqref{add_equation} $c_i=0$ so that \eqref{standard1}
can be achieved directly without going over the $L^{\infty}$ estimate.
For the convenience of the reader we recall this consideration in Section \ref{sec:3} adapting ideas from \cite{HK} to the setting considered here. Note that it seems  that the parameter range to be considered in this paper can not be directly derived from \cite[Thm. 1.2]{HK} by a simple `scaling property' argument. The basic structure of \eqref{tolerance} is here the same as in that reference and the value of $\gamma$ is theoretically computable (but the precise number not that valuable) 
\red{we recall for convenience the outcome for $\delta$ from that reference,
\begin{equation}
\frac{\mu}{3}
-\frac{3}{q}
+
\frac{3}{2}< \delta < \frac{3}{\mu}+\frac{1}{2}+\frac{1}{q}.
\end{equation}}

Moreover, note the little subtleness that the validity of the assumptions in Theorem \ref{main_thm_endversion} does 
not imply that the assumptions of the paper \cite{HK} are satisfied because we fixed therein the 
concrete space of piecewise polynomial functions (while here we keep the space more abstract). 
\item[(iii)] We spend some extra effort to present a realization of the required finite element space, see Appendix \ref{sec:real}.
\end{itemize}

\section{Illustration of the line of arguments in the proof and overview of  different error types}\label{sec:3}
At the beginning we would like to remark that our approach follows 
ideas from \cite{HK} leading to a non-trivial relation between involved parameters (which one might potentially improve in future work). 
We devote the present section to an illustration of the underlying idea as well as its limitations. 

We argue formally. This should serve as supplementing 
information to the independent and rigorous other sections of the paper. 

First, note the very basic fact that actually three  equations are of relevance: the nonlinear, degenerate elliptic and possibly singular level set equation \eqref{unified_equation} with $\varepsilon=0$
which models the original application of the evolution of level sets by a power of the mean curvature, and the regularized equations \eqref{unified_equation} which are also nonlinear,  but not singular and the in the following section introduced linearized equations \red{of type} \eqref{201} with coefficients depending on the solution of the regularized equations. The last two mentioned types of equations also have their discrete counterparts  \eqref{930} and \eqref{L-discrt} in the paper.
The proof is based on a fixed point argument by using the linearized operator.
For it we derive the constant which appears in the error estimate for the linearized equation explicitly. The dependence is implicit since our linearized operator $L$ is not coercive but satisfies a G\r{a}rding inequality, see for the operator $L$ Equation \eqref{2001}. In the literature 
 any standard a priori estimate of type
\beq \label{standard1}
\|u\|_{H^1} \le c \|f\|
\eeq
with suitable norm $\|\cdot\|$
for the equation 
\beq \label{add_equation}
Lu=f
\eeq
relies as far we know usually on an indirect argument (see, e.g., \cite[Cor. 8.7]{MR737190}) from which also the constant $c$ is obtained. Note, that in the different case that $L$ is the Laplacian an explicit constant can be easily derived. In order to get an explicit constant for the general case we test \eqref{add_equation} with $u$ and get by integration by parts formally
\beq
- \int_{\Omega} a_{ij}D_juD_iu + \int_{\Omega} c_iD_iuu =\int_{\Omega} fu.
\eeq
Hence, by using $0 < \lambda \delta_{ij}\le a_{ij}$ (the inequality is understood in the sense of quadratic forms) for some positive $\lambda=\lambda(\varepsilon)$, $\varepsilon$ being the variable appearing in \eqref{2001},  and choosing  $\delta >0$ sufficiently small we have for some $c_{\delta}>0$
\beq \label{add_gradient_estimate}
\lambda \int_{\Omega} |Du|^2 \le \int_{\Omega} (\delta \red{|c_iD_iu|^2}+c_{\delta}u^2)+\int_{\Omega} (f^2+u^2).
\eeq
Now, we use an $L^{\infty}$ estimate of type
\beq
\|u\|_{L^{\infty}} \le c_0 \|f\|
\eeq
with constant $c_0$ whose dependence on the regularization parameter $\varepsilon$ can be derived explicitly. Via
\eqref{add_gradient_estimate} this can be put together immediately to \red{obtain} an $H^1$ a priori estimate for $u$
with an explicit constant.
Note that this trick requires that the discrete estimates are  always derived by the $H^1$ estimate while the norms of interest are of type $W^{2,\mu}$ and hence, the required  inverse estimates reduce the powers of $h$.
Proceeding in this way it remains as a challenge to achieve a  less restrictive range for all the technical parameters in the main result Theorem \ref{main_thm_endversion}, especially to get the largest possible power of $h$ in the definition of $\rho$. We think that obtaining {\it any} error estimate with qualitative dependencies on the parameters is already challenging.

Second, the asymptotic \red{value} of $c_0$ in terms of $1/\varepsilon$ is of interest. The linearized equation evaluated at $u^{\varepsilon}\in C^{\infty}$ for the case that the gradient of $u^{\varepsilon}$ is small in an open subset of the domain and small compared to $\varepsilon$, has  coefficients (of the lower order (e.g. first order) terms) locally of size $\frac{1}{\varepsilon^m}$, $m\in \mathbb{N}$, and arbitrary sign.
 A simple prototype equation of that type with solution $u=u(t)$,
 \beq
 -u''+\frac{1}{\varepsilon}u'=0
 \eeq
 in $(0,1)$, already shows that $u' \approx c e^{\frac{t}{\varepsilon}}$ and \red{$u'' \approx  \frac{c}{\varepsilon}e^{\frac{t}{\varepsilon}}$}.
 This is a strong indication for optimality of our exponential bound in Theorem \ref{main_thm_endversion}
 as long as this worst case prototype can appear.
It has not been studied yet if such worst cases can be excluded. We leave this point  for future work. If that is possible this implies that in the current paper and previous work \cite{HK} the exponential asymptotic in the main theorem can be replaced by a polynomial one.
 %
 It is quite plausible that the quality of the overall error bound from our main theorem is prescribed by the `quality' of the constant in the a priori estimate for the linearized equation
 \eqref{201}. 
 Experimentally, a polynomial coupling between $\varepsilon$ and $h$ (of size $h=\varepsilon^2$) is suggested in \cite[p.114]{MR2350186} which refers to our knowledge to observations from an example calculation which excludes singular points (i.e. stationary points of the solution) in the computational domain.
 \red{In our experimental previous work}, in a more general case, i.e. with stationary points in the computational domain, we even could not solve computationally the problem with very small values of $\varepsilon$, see~\cite{MR3892432}. From the parameter regime where computations were possible we cannot draw a clear conclusion for a suggestion of a rate.
The parabolic approach in \cite{MR1417861} and \cite{MR1760409} (where especially in the first paper a different class of equations is considered which also includes mean curvature flow) allows error estimates of even polynomial type in inverse powers of the regularization parameter. Note that we study an elliptic equation and that we are not aware of a parabolic level set approach to the inverse mean curvature flow.

Third, recalling the Sobolev embedding theorem, see Appendix \ref{app:C} we see that here we need second weak derivatives in order to have embeddings at least in $C^1$, since for obtaining $k\ge 1$ we need by \eqref{sob-2}  that $m\ge 2$. In this case the embedding is given by $W^{2,p}(\Omega) \subset C^{1,\alpha}(\Omega)$ continuous with $ 3< p <\infty$ and $\alpha:=1-3/p$.

Fourth, in Table \ref{table:kysymys} we present open problems and proved results towards a full error estimate up to the $C^1$ norm.
The table gives an overview about literature on different errors appearing in connection with our scheme (note that for future work it might be worth studying and comparing with schemes
which omit the regularization procedure).  For $u$ being the solution of  \eqref{unified_equation} with $\varepsilon=0$, $u^{\varepsilon}$ the solution of \eqref{unified_equation} with $\varepsilon>0$, and $u_h^{\epsilon}$ the solution of~\eqref{930} the corresponding full errors, regularization errors, and discretization errors
 of type~1 (i.e. with explicit dependence on $\varepsilon$) and type 2 (i.e. without explicit dependence on $\varepsilon$) are shown. Here $0<\alpha <1$ is arbitrary
apart from \cite{MR3892432} where $\alpha=0$ and $u$ and $u^{\varepsilon}$ are specified by the more complicated model corresponding to the inverse mean curvature flow.
Concerning the $C^1$-norm row we remark that when one does not take care about explicit constants the estimate of type 1  which serves here as first estimate of type 2 can be improved (in the sense of a type 2 estimate) concerning the powers of $h$. This can be done by classical arguments, however an explicit derivation of the dependence on the parameters might be unclear. In the two cases marked as open even the convergence is unknown and is probably  only under stronger assumptions on $\Omega$ expectable, e.g. convexity. Note that this table ignores the additional complexity arising from modeling the inverse mean curvature flow by using also the domain and the boundary values as parameters. 
\begin{table}[htb]
\resizebox{0.6\textwidth}{!}{
\begin{tabular}{lcccc}
\toprule
& \multicolumn{2}{c}{$||u^{\epsilon}-u_h^{\epsilon}||$} &    &  \\
\cmidrule{2-3}
&type 2 & type 1 &  $||u-u^{\epsilon}||$ &  $||u-u_h^{\epsilon}||$ \\ 
\cmidrule{2-2}\cmidrule{3-3}\cmidrule{4-4}\cmidrule{5-5}
\vspace{1mm}
$C^{0,\alpha}$-norm & \cite{MR2350186} &  \cite{HK} &\cite{MR3892432,MR3734408} & \cite{HK}  \\
$C^1$-norm &   follows from &  Thm. \ref{main_thm_endversion}  & open& open \\
 &   type 1&  ($\mu>3$) &  &  \\
\bottomrule
\end{tabular}}
\vspace{2mm}
\caption{Literature on different error estimates.}
\label{table:kysymys}
\end{table}

Fifth, we present results from the literature under which conditions the solution (also called `arrival time') of the equation \eqref{unified_equation} in the case $k=1$
and $\varepsilon=0$ (this is the most important case of mean curvature flow) has higher regularity.
 \red{In \cite{MR1216584,MR1030675} it is shown} that for a convex initial hypersurface the flow is smooth except at the point where it becomes extinct and that the arrival time is then $C^2$.
 \cite{MR1189906,MR1216585} present an example in $\mathbb{R}^3$ of a rotationally symmetric mean convex dumbbell with arrival time being not $C^2$.
 \cite{MR2200259} shows that it is at least $C^3$ in the case of curves. \cite{MR2383930} shows for $n > 1$ that the above mentioned result from  \cite{MR1216584,MR1030675} is optimal by giving examples of convex initial hypersurfaces with arrival time being not three times differentiable.
\cite{MR3772402} proves: If $\partial \Omega$ has positive  mean curvature, then the arrival time $u$ is $C^2$ iff the following two conditions hold:

\begin{itemize}
\item[(i)] There is exactly one singular time T (where the flow becomes extinct).

\item[(ii)] The singular set (i.e. the set of critical points) is a $k$-dimensional closed, connected, embedded $C^1$ submanifold of singularities where the blowup is a cylinder $S^{n-k}\times \mathbb{R}^k$ at each point.
\end{itemize}

Finally, in view of the higher regularity of the level set function in special cases
it seems realistic to expect also convergence of the regularization error in higher order norms in special cases which we leave here open.


\section{Discrete $W^{2,\mu}$-estimates with explicit constants
} \label{higher_order_estimates}
We derive some auxiliary estimates for a class of linear equations which includes in particular the linearization of \eqref{unified_equation}; we proceed analogously to \cite{HK}. 

\if{To prove these auxiliary estimates we proceed similarly as in \cite{HK},
apart from when the $L^{\infty}$--estimate comes into play then one has to use the estimate \eqref{new_sup_estimate-est}. We omit here the detailed proofs
and refer for details to \cite{HK} but state all lemmas for the convenience of the reader in the new variant.}\fi

The linearized operator $L_{\varepsilon}:=L_{\varepsilon}(u^{\varepsilon})\colon H^1_0(\Omega)\rightarrow H^{-1}(\Omega)$  for  \eqref{unified_equation} in a solution $u^{\varepsilon}$
is given by
\beq \label{2001}
Lu := D_i(a_{ij}D_ju) + c_iD_i u
\eeq
in $\Omega$ with coefficients given as follows (such a calculation can be found in \cite{HK}):
We define for $\varepsilon>0$ and $z \in \mathbb{R}^3$
\beq
|z|_{\varepsilon}:=f_{\varepsilon}(z):=\sqrt{|z|^2+\varepsilon^2}
\eeq
and denote derivatives of $f_{\varepsilon}$ with respect to $z_i$ by $D_{z_i}f_{\varepsilon}$, i.e. 
there holds
\beq
D_{z_i}f_{\varepsilon}(z)= \frac{z_i}{|z|_{\varepsilon}}; \quad 
D_{z_i}D_{z_j}f_{\varepsilon}(z)=\frac{\de_{ij}}{|z|_{\varepsilon}} - \frac{z_iz_j}{|z|^3_{\varepsilon}};
\eeq
with these notations we set
\begin{equation}
a_{ij} := -D_{z_i}D_{z_j}f_{\varepsilon}(Du^{\varepsilon}) \quad \text{and} \quad  c_i:=\eta'(|Du^{\varepsilon}|)D_{z_i}f_{\varepsilon}(Du^{\varepsilon}).
\end{equation}
Straightforward  application  of the Cauchy-Schwarz inequality  shows the positive definiteness of $a_{ij}$,
\begin{equation}
\left(
|z|_{\varepsilon}^2\delta_{ij}-z_iz_j
\right)
\eta_i\eta_j
\ge |z|_{\varepsilon}^2|\eta|^2-\left<z, \eta\right>^2 
= (|z|^2+\varepsilon^2)|\eta|^2-\left<z, \eta\right>^2\ge \varepsilon^2|\eta|^2
\end{equation}
for any $\eta=(\eta_i) \in \mathbb{R}^n$ and  $z=(z_i)\in \mathbb{R}^n$.
Furthermore, we introduce the equation
\beq \label{201}
Lu = g+D_i f_i
\eeq
with $f_i\in W^{1,p}(\Omega)$, $1\le i\le 3$, and $g$ in $L^p(\Omega)$, $p\ge 1$. The special formal structure of the right-hand side
is chosen as in \cite{HK} due to the use of $H^1$ conforming finite elements one only has $f_i\in L^p(\Omega)$.

\if{For the convenience of the reader we state the following well-known 
 inverse estimate for finite elements  from \cite{MR1278258} which will be used several times without mentioning it explicitly each time.
 \footnote{Uniform triangulation important. AK} }\fi

 For convenience we recall the inverse estimate which will be used without mentioning it each time.
\begin{lemma}[Inverse estimate] \label{inverse_estimate_lemma}
For $1\le p,q\le \infty$ and $0\le m\le l\le 2$ there exists 
a constant $c>0$ such that
\begin{equation}\label{inv-est}
\|v_h\|_{W^{l,p}(\Omega^h)} \le c h^{m-l+\min(0, \frac{3}{p}-\frac{3}{q})}\|v_h\|_{W^{m,q}(\Omega^ h)}
\end{equation}
for all $v_h \in V_h$.
\end{lemma}
\begin{proof}
  The result follows from the inverse estimate \cite[Sec. 4.5]{MR1278258} and Remark \ref{app:inv}. 
\end{proof}

\red{In the sequel $P$ applied to a list of positive arguments $l_1, \dots, l_m$, $m\in \mathbb{N}$, stands for 
an expression which depends at most polynomial on $\max\{l_i, 1/l_i\}$, $i=1, \dots, m$; i.e. we have 
with $r \in \mathbb{N}$ and $c>0$ that
\begin{equation} \label{explain_P}
P(1/\varepsilon)\le \mathcal{O}(\varepsilon^{-r})
\end{equation}
for small $\varepsilon$ 
where $c,r$ do not depend on $\varepsilon$ and may vary  from line to line where the expression appears.}

Using standard elliptic regularity theory it is shown in \cite[Sec. 5]{HK} that  
the solution $u^{\varepsilon}$ of \eqref{unified_equation} is smooth and satisfies the estimate
\beq
\|u^{\varepsilon}\|_{H^m(\Omega)} =P(1/\varepsilon)
\eeq
for all $m\in \mathbb{N}$.

 Writing $L=L_{\epsilon}$ we say that $u_h \in V_h$ is a finite element solution of \eqref{201} with homogeneous Dirichlet boundary conditions if
\beq\label{L-discrt}
- \int_{\Omega^h} a_{ij}D_ju_h D_i v_h \dd x + \int_{\Omega^h} c_iD_i u_h  v_h \dd x = \int_{\Omega^h} (D_if_i+g) v_h \dd x \quad \forall v_h \in V_h.
\eeq


\begin{lemma}\label{220}
Let $\delta_0$ be given as in Section \ref{sec:main_result} and $h_0>0$ such that \eqref{h_0} is satisfied. We allow
unless specified concretely that
\begin{equation}
\left\{
 \begin{aligned}
\tilde \Omega&=\Omega^h\text{ with $0<h<h_0$ or }\\
\tilde \Omega&=\Omega_{\de_0} 
 \end{aligned}
 \right.
\end{equation}
and  assume $D_if_i+g \in L^3(\tilde \Omega)$.

(i) There exists a unique solution $u \in H^2(\tilde \Omega)\cap H^1_0(\tilde  \Omega)$ of \eqref{201} with $L=L_{\epsilon}$ satisfying
\bea \label{241-b}
\norm{u}{H^2( \tilde \Omega)}&\le  e^{ P(1/\varepsilon)}\|D_if_i+g\|_{L^3(\tilde \Omega)}.
\eea

(ii) Furthermore, assuming $\tilde \Omega=\Omega^h$ in (i) we have the best-approximation property
\bea\label{abc}
\norm{u - u_h}{H^1(\Omega^h)}\le e^{ P(1/\varepsilon)}\inf_{v_h \in V_h} \norm{u-v_h}{H^1(\Omega^h)}
\eea
where $u_h$ is a solution of \eqref{L-discrt}.
\end{lemma}
\begin{proof}
(i) First note that existence and uniqueness of a solution $u$ follows from classical results in PDE theory.
These can be found in standard textbooks and more specifically adapted for us 
in \cite[Lemma 1]{MR2350186}. Note that the different exponent in the first order derivative term and its different sign in that reference compared to 
our case does not play a role concerning applicability of that result. Furthermore, $\tilde \Omega$ has in both cases a  $C^2$-regular boundary.

As an intermediate step we prove
\bea \label{241}
\|Du\|_{L^2(\tilde \Omega)}&\le  e^{ P(1/\varepsilon)}\|D_if_i+g\|_{L^3(\tilde \Omega)}.
\eea
This estimate follows from \cite[Lem. 7.1]{HK} replacing \cite[Thm. 6.2]{HK} therein by the $L^{\infty}$--estimate in
  Theorem~\ref{new_sup_estimate}. There, $D_if_i+g$ from our setting plays the role of $g$ in \cite[Lem.~7.1]{HK}, note that
  the two situations differ since we assume now higher regularity for $f_i$. 
Estimate \eqref{241-b} is a straightforward calculation by combining the standard proof
for higher regularity, see \cite{MR737190},  with the estimate~\eqref{241}.

(ii) Existence and uniqueness of $u_h$ follows as in the proof of \cite[Lemma 2]{MR2350186} by using the Schatz argument. Note that our finite element ansatz space is especially $H^1$ conforming (since even $H^2$ conforming) so that standard theory for $H^1$ conforming finite element theory carries over to our setting.

The proof follows exactly the lines in \cite[estimate~(7.12)]{HK} using  \eqref{241-b} instead of \cite[(7.10)]{HK}; we remark that in the latter reference cubic elements are used, however, since here we have ansatz functions of higher polynomial degree and a boundary approximation as given in \eqref{851}, the situation here is at least as convenient as before and allows the application of the former arguments. 
\end{proof}

\begin{theorem} \label{221}
Let $u$ be the unique solution of  \eqref{201} in $\Omega^h$ where $L=L_{\epsilon}$.
Then, there is $h_0>0$ so that for $0<h \le h_0$ 
there exists a unique finite element solution $u_h\in V_h$ of  \eqref{L-discrt} in $\Omega^h$ satisfying
\beq
\|u_h\|_{H^2(\Omega^h)} \le e^{ P(1/\varepsilon)}\|D_if_i+g\|_{L^3(\Omega^h)}.
\eeq

\if{
\beq \label{240}
\|u-u_h\|_{H^2(\Omega^h)} \le e^{ P(1/\varepsilon)} \inf_{v_h\in V_h}\|u-v_h\|_{H^2(\Omega^h)}
\eeq
and 
\beq \label{242}
\|u_h\|_{H^2(\Omega^h)} \le e^{ P(1/\varepsilon)} \|u\|_{H^2(\Omega^h)}.
\eeq
}\fi
\end{theorem}
\begin{proof}
\if{Let $\mathcal{I}_h\colon H^2(\Omega)\rightarrow  V_h$ be the nodal interpolant (XXX); by the embedding $H^2(\Omega) \subset C^0(\Omega)$ it is well-defined (XXX Dimension).
}\fi
Existence and uniqueness of the solution $u$ follows from Lemma \ref{220} (i).
We recall that $\mathcal I_h$ is the interpolation operator introduced in Assumption~\ref{hyp2}.
Setting \red{$e_h:=u-u_h$} we have 
\bea
 \norm{e_h}{H^2(\Omega^h)} &\le \norm{u -\mathcal{I}_h u}{H^2(\Omega^h)} +\norm{u_h -\mathcal{I}_h u}{H^2(\Omega^h)}\\
 & \le \norm{u -\mathcal{I}_h u}{H^2(\Omega^h)} + ch^{-1} \left(\norm{u -\mathcal{I}_h u}{H^1(\Omega^h)} + \norm{e_h}{H^1(\Omega^h)} \right) \\
 & \le \norm{u -\mathcal{I}_h u}{H^2(\Omega^h)} +  c \norm{u }{H^2(\Omega^h)} + ch^{-1}\norm{e_h}{H^1(\Omega^h)}.
 \eea
 Since by \eqref{abc}
 \bea
 \norm{e_h}{H^1(\Omega^h)} \le e^{ P(1/\varepsilon)} \norm{u - \mathcal{I}_h u}{H^1(\Omega^h)},
 \eea
 we have by interpolation estimate~\eqref{866} that
 \bea
 \norm{e_h}{H^2(\Omega^h)} &\le  e^{ P(1/\varepsilon)} \norm{u}{H^2(\Omega^h)}
 \eea
 and we conclude by the triangle inequality and stability estimate \eqref{241-b}.

\if{See the proof of \cite[Thm. 7.3]{HK} and use therein Lemma \ref{220} instead of the analogous \cite[Lem.~7.1]{HK} in order to obtain the polynomial constants $e^{ P(1/\varepsilon)}$.
This leads to \eqref{240} and \eqref{242} with $H^2(\Omega^h)$--norms replaced by $H^1(\Omega^h)$--norms at every place. In order to show \eqref{240} and \eqref{242} one has
}\fi
\end{proof}

\if{
\begin{remark} \label{823}
In the situation of Theorem \ref{221} holds
\beq
\|u_h\|_{H^2(\Omega^h)} \le e^{ P(1/\varepsilon)}\|D_if_i+g\|_{L^3}.
\eeq
\end{remark}
\begin{proof}
This follows immediately from Theorem \ref{221}, cf. also with the proof of \cite[Corollary 7.4]{HK}.
\end{proof}
}\fi


\section{Banach's fixed point theorem in $W^{2, \mu}$ balls with radii given explicitly in terms of $h$ and $\varepsilon$} \label{Section_FE}

To obtain existence and uniqueness of a solution $u^{\varepsilon}_h \in \bar B^h_{\rho}$  (defined in \eqref{912}) of \eqref{930} we proceed as follows: By identifying this solution with the unique fixed point of the mapping
$T: V_h\rightarrow V_h$ defined by
\beq \label{204}
L_{\varepsilon} (w_h-T(w_h)) = \Phi_{\varepsilon}(w_h), \quad w_h \in V_h
\eeq
with the operator $\Phi_{\varepsilon}$ given by
\begin{align}
 \Phi_{\varepsilon}\colon H^1_0(\Omega^h) \rightarrow H^{-1}(\Omega^h),\quad \Phi_{\varepsilon}(v):=-D_i\left(\frac{D_i v}{|Dv|_{\varepsilon}}\right) + \eta (|Dv|_{\varepsilon})
\end{align}
and where we understand equation \eqref{204} (and in the sequel analogous equations) in the finite element solution sense according to \eqref{L-discrt}. 
Uniqueness and existence of the fixed point in $ B^h_{\rho}$ follows by Banach's fixed point theorem. We will check the standard assumptions of the fixed point theorem in a quantitative way with respect to constants, more precisely the following selection of three sufficient conditions:

\textbf{(i)} Non-emptyness:
\beq \label{913}
\bar B^h_{\rho} \neq \emptyset.
\eeq

\textbf{(ii)} Contraction property: For $\mu \ge 2$  and some $\eta>0$ we have
\beq \label{914}
\begin{aligned}
\|T(w_h)&-T(v_h)\|_{W^{2, \mu}( \Omega^h)} \le e^{P(1/\varepsilon)} h^{\eta} \|w_h-v_h\|_{W^{2, \mu} (\Omega^h)}  \quad \forall w_h, v_h \in \bar B^h_{\rho}.
\end{aligned}
\eeq

\textbf{(iii)} Self-mapping:
\beq \label{915}
T({\bar B^{h}_{\rho}}) \subset {\bar B^{h}_{\rho}}.
\eeq


Next, we check conditions \textbf{(i)} to \textbf{(iii)}: 

\textbf{(i)}: \red{By interpolation estimate \eqref{866} for given $\rho$ one can always find an $\tilde{h}>0$ such that $\mathcal{I}_{\tilde{h}} u^{\varepsilon} \in  B^h_{\rho}$} (since $\delta < \deg-2$).

\if{The condition follows from our assumed sufficiently good interpolation operator
\begin{equation}
\mathcal{I}_h:H^{\deg}(\bar \Omega^h)\rightarrow \tilde V_h,
\end{equation}
namely (see \cite[equation (1.103)]{MR2050138})
\beq \label{866}
\|u-\mathcal{I}_hu\|_{W^{{\deg}-1-m,\infty}(\Omega^h)} \le ch^{m+1}\|u\|_{H^{{\deg}}(\bar \Omega^h)} \quad \forall \; u \in H^{{\deg}}(\bar \Omega ^h)
\eeq
where $m =1, 2$. 
}\fi
\textbf{(ii)}: We show an estimate for $\|T(v_h)- T(w_h)\|_{W^{2,\mu}(\Omega^h)}$ with $\mu\ge 2$. 
Let $v_h$ and $w_h$ be in $\bar B^{h}_{\rho}$, $\xi_h = v_h-w_h$, $\al(t)=w_h + t\xi_h$, $0\le t \le 1$. In view of \eqref{204} we have
\beq \label{501}
L_{\varepsilon} (T(v_h)-T(w_h))  = L_{\varepsilon}\xi_h + \Phi_{\varepsilon}(w_h)-\Phi_{\varepsilon}(v_h).
\eeq
 Recalling the convention that when $\eta$ and $f_{\varepsilon}$ have no arguments it is meant $ \eta=\eta(|Du^{\varepsilon}|_{\varepsilon})$ and $f_{\varepsilon}=f_{\varepsilon}(D u^{\varepsilon})$ the right-hand side of \eqref{501} is of the form \red{$D_if_{i}+g\in L^3$} with (see \cite{HK})
\bea \label{868}
f_i &= D_{z_i}f_{\varepsilon}(D v_h)-D_{z_i}f_{\varepsilon}(D w_h)-D_{z_i}D_{z_m}f_{\varepsilon}D_m \xi_h \\
&= \int_0^1\left(D_{z_m}D_{z_i}f_{\varepsilon}(D \al(t))- D_{z_i}D_{z_m}f_{\varepsilon}\right) D_m \xi_h
\eea
and (using the convention that $\eta'$ denotes the derivative of the function $r \mapsto \eta(r)$
and that the argument of $f_{\varepsilon}$ if omitted is understood to be $Du^{\varepsilon}$) we have
\bea \label{867}
g&=\eta'D_{z_m}f_{\varepsilon}D_m \xi_h + \eta\left(f_{\varepsilon}(D w_h)\right)- \eta\left(f_{\varepsilon}(D v_h)\right)\\
&= \int_0^1 \left(\eta'D_{z_m}f_{\varepsilon}-
\eta'\left(f_{\varepsilon}(D \al(t))\right)D_{z_m}f_{\varepsilon}(D \al(t))\right)D_m \xi_h.
\eea
Since the finite element space is $H^2$--conforming we may rewrite $D_if_i$
by performing the differentiation and get with the abbreviation
\bea
G(t):= D_{z_r}D_{z_i}f_{\varepsilon}(D\alpha(t))D_iD_r\alpha(t)
\eea
that
\bea \label{4.12}
D_{i}f_i&=
D_{z_r}D_{z_i}f_{\varepsilon}(D\alpha(1))D_iD_r\alpha(1)
-D_{z_r}D_{z_i}f_{\varepsilon}(D\alpha(0))D_iD_r\alpha(0)\\
&\quad - D_{z_r}D_{z_i}D_{z_m}f_{\varepsilon}D_{i}D_{r}u^{\varepsilon}D_m\xi_h
- D_{z_i}D_{z_m}f_{\varepsilon}D_iD_m\xi_h \\
&= \int_0^1\frac{\dd}{\dd t}G(t)- D_{z_r}D_{z_i}D_{z_m}f_{\varepsilon}D_m\xi_hD_{i}D_{r}u^{\varepsilon}- D_{z_i}D_{z_m}f_{\varepsilon}D_iD_m\xi_h.
\eea
Moreover, we have
\bea\label{510-1}
\frac{\dd}{\dd t}G(t)&= D_{z_m}D_{z_r}D_{z_i}f_{\varepsilon}(D\alpha(t))
D_m\xi_hD_iD_r\alpha(t)\\
&\quad + D_{z_r}D_{z_i}f_{\varepsilon}(D\alpha(t))D_iD_r\xi_h.
\eea
We rewrite the first term in $\frac{\dd}{\dd t}G(t)$ by using the identity
\bea
ABC=(A-a)(B-b)C+(A-a)bC+a(B-b)C+abC
\eea
for real numbers $A, B, C, a, b$. With 
\begin{equation}
\begin{aligned}
A&:= D_{z_m}D_{z_r}D_{z_i}f_{\varepsilon}(D\alpha(t));& B&:=D_iD_r\alpha(t); \quad C:= D_m\xi_h;\\
a&:=D_{z_m}D_{z_r}D_{z_i}f_{\varepsilon}(Du^{\varepsilon}) ;& b&:= D_iD_ru^{\varepsilon}
\end{aligned}
\end{equation}
 we obtain
\bea \label{expansion_G}
\frac{\dd}{\dd t}G(t)&= \left(D_{z_m}D_{z_r}D_{z_i}f_{\varepsilon}(D\alpha(t))
-D_{z_m}D_{z_r}D_{z_i}f_{\varepsilon}(Du^{\varepsilon})\right)
(D_iD_r\alpha(t)-D_iD_ru^{\varepsilon})D_m\xi_h\\
&\quad + \left(D_{z_m}D_{z_r}D_{z_i}f_{\varepsilon}(D\alpha(t))
-D_{z_m}D_{z_r}D_{z_i}f_{\varepsilon}(Du^{\varepsilon})\right)
D_iD_ru^{\varepsilon}D_m\xi_h\\
&\quad + D_{z_m}D_{z_r}D_{z_i}f_{\varepsilon}(Du^{\varepsilon})
(D_iD_r\alpha(t)-D_iD_ru^{\varepsilon})D_m\xi_h\\
&\quad + D_{z_m}D_{z_r}D_{z_i}f_{\varepsilon}(Du^{\varepsilon})
D_iD_ru^{\varepsilon}D_m\xi_h \\
&\quad + \left(D_{z_r}D_{z_i}f_{\varepsilon}(D\alpha(t))-D_{z_r}D_{z_i}f_{\varepsilon}(Du^{\varepsilon})\right)D_iD_r\xi_h + 
D_{z_r}D_{z_i}f_{\varepsilon}(Du^{\varepsilon})D_iD_r\xi_h.
\eea
Due to cancellations with terms in \eqref{510-1} and commutation of derivatives we obtain
\bea \label{expansion_Difi}
D_if_i&= \int_0^1\bigg[\left(D_{z_m}D_{z_r}D_{z_i}f_{\varepsilon}(D\alpha(t))
-D_{z_m}D_{z_r}D_{z_i}f_{\varepsilon}(Du^{\varepsilon})\right)
D_m\xi_h(D_iD_r\alpha(t)-D_iD_ru^{\varepsilon})\\
&\quad + \left(D_{z_m}D_{z_r}D_{z_i}f_{\varepsilon}(D\alpha(t))
-D_{z_m}D_{z_r}D_{z_i}f_{\varepsilon}(Du^{\varepsilon})\right)
D_m\xi_hD_iD_ru^{\varepsilon}\\
&\quad + D_{z_m}D_{z_r}D_{z_i}f_{\varepsilon}(Du^{\varepsilon})
D_m\xi_h(D_iD_r\alpha(t)-D_iD_ru^{\varepsilon})\\
&\quad + \left(D_{z_r}D_{z_i}f_{\varepsilon}(D\alpha(t))-D_{z_r}D_{z_i}f_{\varepsilon}(Du^{\varepsilon})\right)D_iD_r\xi_h\bigg]. 
\eea

\red{
To estimate the $W^{2,\mu}(\Omega^h)$--norm of $u_h:=T(v_h)- T(w_h)$ we proceed as follows: As an upper bound  we obtain the product of the $H^2(\Omega^h)$--norm of $u_h$ and some negative powers of $h$ by applying inverse estimates. Then using Theorem \ref{221} applied to equation \eqref{501} with  $g$ and $D_if_i$ as expressed in \eqref{867} and \eqref{expansion_Difi} we can further estimate $\norm{u_h}{H^2(\Omega^h)}$ from above in terms of the $L^3(\Omega^h)$ norm of $g+D_if_i$ which can be further estimated by the triangle inequality, 
 the interpolation estimate \eqref{866}, and the inverse estimate \eqref{inv-est}.
We begin this with an auxiliary estimate stated in the following.}
\red{
 Setting 
\begin{align}
 \pi(h,\varepsilon):=h^{-\frac{3}{\mu}} (\rho + h^{{\deg}-2}\|u^{\varepsilon}\|_{C^{{\deg}}(\bar \Omega^h)})
+h^{{\deg}-2}\|u^{\varepsilon}\|_{C^{{\deg}}(\bar \Omega^h)}
\end{align}
we have
\if{\begin{equation} \label{5.15}
\begin{aligned}
\|w_h-u^{\varepsilon}\|_{W^{2,\infty}(\Omega^h)}  \le  c\pi(h,\varepsilon).
\end{aligned}
\end{equation}
}\fi
}
\red{
\begin{equation} \label{8.16}
 \begin{aligned}
\|w_h-u^{\varepsilon}\|_{W^{2,\infty}(\Omega^h)} 
&\le
\|w_h-\mathcal{I}_h u^{\varepsilon}\|_{W^{2,\infty}(\Omega^h)}+\|\mathcal{I}_hu^{\varepsilon}-u^{\varepsilon}\|_{W^{2,\infty}(\Omega^h)}  \\
&\le   ch^{-\frac{3}{\mu}}\left(\|w_h-u^{\varepsilon}\|_{W^{2,\mu}(\Omega^h)}
+\|u^{\varepsilon}-\mathcal{I}_h u^{\varepsilon}\|_{W^{2,\mu}(\Omega^h)}\right) \\
& \quad +ch^{{\deg}-2}\|u^{\varepsilon}\|_{C^{{\deg}}(\bar \Omega^h)}\\
&\le c\pi(h,\varepsilon).
\end{aligned}
\end{equation}
where we used in this order: the triangle inequality, the inverse estimate \eqref{inv-est} together with the triangle inequality, and then the  interpolation estimate  \eqref{866}.
Note that clearly the same inequality holds with $w_h$ replaced by $v_h$.}
}

\if{Clearly, since $\rho$ plays the role of the radius of the ball $\bar B_{\rho}^h$ in which we confirm the assumptions of Banach's fixed point theorem (and we are interested in the asymptotic regime of the parameter values) it will naturally have a small value so that we may assume w.l.o.g. its boundedness by a moderate constant, e.g. $\rho<1$, which will be used for the following tacitly (and will vanish in generic constants).}\fi

\red{
Next, we define
\begin{equation}
\begin{aligned}
d_{v,1}&:= D(v_h-u^{\varepsilon}),&
d_{w,1}&:= D(w_h-u^{\varepsilon}),\\
d_{v,2}&:= D^2(v_h-u^{\varepsilon}),&
d_{w,2}&:= D^2(w_h-u^{\varepsilon}).
\end{aligned}
\end{equation}
}

\red{
With \eqref{expansion_Difi} we obtain
\begin{equation} \label{estmate_f}
 \begin{aligned}  
\|D_if_i\|_{L^3(\Omega^h)}
 & \le  2e^{ P(1/\varepsilon)} \Pi_{i=1,2}
 (\|d_{w,i}\|_{L^{\infty}(\Omega^h)}+\|d_{v,i}\|_{L^{\infty}(\Omega^h)})  \|D\xi_h\|_{L^3(\Omega^h)} \\
&\quad + e^{ P(1/\varepsilon)}(\|d_{w,1}\|_{L^{\infty}(\Omega^h)}
+\|d_{v,1}\|_{L^{\infty}(\Omega^h)})
\|D^2\xi_h\|_{L^3(\Omega^h)}.
 \end{aligned}
\end{equation}
Furthermore, we have
\begin{equation} \label{estmate_g}
\|g\|_{L^3(\Omega^h)} \le e^{ P(1/\varepsilon)} \left( \|d_{w,1}\|_{L^{\infty}(\Omega^h)}+\|d_{v,1}\|_{L^{\infty}(\Omega^h)}\right)\|D\xi_h\|_{L^3(\Omega^h)}.
\end{equation}
We estimate the right-hand sides of \eqref{estmate_f} and \eqref{estmate_g} further from above in terms of $\rho$, $h$ and $\varepsilon$
which will \red{be} done by relating the appearing $L^3(\Omega^h)$ norms to $L^{\mu}(\Omega^h)$ norms.
More precisely, we apply  \eqref{8.16} and obtain
\begin{equation} \label{5.20}
\begin{aligned}
\|D_if_i \|_{L^3(\Omega^h)} & \le 
e^{ P(1/\varepsilon)} \pi(h,\varepsilon)h^{\min\left(0,1-\frac{3}{\mu}\right)} \norm{\xi_h}{W^{2,\mu}(\Omega^h)} =: A  \norm{\xi_h}{W^{2,\mu}(\Omega^h)}
\end{aligned}
\end{equation}
where we assume that $\pi(h,\varepsilon)<1$; this will be automatically the case due to the stronger inequality \eqref{sufficient_contraction}.
}
\red{
We observe that  
 $ \|g\|_{L^3(\Omega^h)}$ can be estimated from above also by the right-hand side of~\eqref{5.20}.
By Theorem \ref{221}, an inverse estimate, and using $\mu\ge 2$ we obtain
\begin{equation} \label{870}
 \begin{aligned}
  \|T(v_h)- T(w_h)\|_{W^{2,\mu}(\Omega^h)} &\le h^{\frac{3}{\mu}-\frac{3}{2}} \|T(v_h)- T(w_h)\|_{H^2(\Omega^h)} \le \omega \norm{\xi_h}{W^{2,\mu}(\Omega^h)}
 \end{aligned}
\end{equation}
with
\begin{align}
  \omega:= e^{ P(1/\varepsilon)}h^{\frac{3}{\mu}-\frac{3}{2}} A.
\end{align}
Sufficient for $T$ to satisfy \eqref{914} is that 
\begin{align}\label{sufficient_contraction}
\omega <1 \quad \text{for small $h$}.
\end{align}
We get as sufficient conditions for 
\eqref{sufficient_contraction}  that 
\begin{equation}\label{suff1a}
\begin{aligned}
\delta &> 3/2 &&\text {if } \mu > 3, \\
\delta &> 1/2+3/\mu&&\text {if } 2 \le \mu \le 3
\end{aligned}
\end{equation}
implying the contraction property \eqref{914}.
Note that $\delta$ is the power of $h$ in the definition of $\rho$ and that $\operatorname{deg}\ge 8$ by Remark \ref{rem:101}; we also used $\operatorname{deg}>2+\delta$, which will be replaced by a stronger condition later in the proof.
}

\textbf{(iii)}: 
\red{We recall the definition of $V_h$ and $\tilde{V}_h$ given in and after Assumption \ref{hyp1}, as well as that $N_h$ denotes the set of nodes in $\mathcal{T}_h$.}
In order to prove \eqref{915}
we choose $z_h \in \tilde V_h$  (in a not unique way) by requiring that 
\beq \label{z_h_defi}
z_h = \begin{cases}\mathcal{I}_hu^{\varepsilon} \quad &\text{in } \partial \Omega^h, \\
0 \quad &\text{in } N_h \backslash \partial \Omega^h \end{cases}
\eeq
and
\beq \label{z_h_2}
  \|z_h\|_{L^{\infty}(\Omega^h)}\le c \|{\mathcal{I}_h u^{\varepsilon}}_{|\partial \Omega^h}\|_{L^{\infty}(\partial \Omega^h)}\le {P(1/\varepsilon)}h^{\widetilde \deg},
\eeq
where the last inequality follows from a first order Taylor's expansion of $u^{\varepsilon}$ in view of the boundary approximation property, cf. Assumption~\ref{hyp2}.
We set
\beq \label{883}
\tilde u^{\varepsilon}:= \mathcal{I}_hu^{\varepsilon}-z_h.
\eeq
Then $\tilde u^{\varepsilon}\in V_h$ and for all $1\le q \le \infty$ 
\beq \label{901}
\|\tilde u^{\varepsilon}-u^{\varepsilon}\|_{W^{2,q}(\Omega^h)} \le c h^{{\widetilde \deg} -2+\frac{\widetilde \deg}{q}}\|u^{\varepsilon}\|_{C^{\deg}(\bar \Omega^h)} 
\eeq
which follows from the standard interpolation error estimate and the consideration at the boundary. 
For it  we repeat on the zero order level the estimate for $z_h$ from above, cf. \eqref{z_h_2}, and, furthermore, derive by using inverse estimates the following estimates
\begin{equation}
\begin{aligned}
\|z_h\|_{L^{\infty}( \Omega^h)} &\le {P(1/\varepsilon)}h^{{\widetilde \deg}},\\
\|Dz_h\|_{L^{\infty}(\Omega^h)} &\le {P(1/\varepsilon)} h^{{\widetilde \deg} -1},\\
\|D^2z_h\|_{L^{\infty}(\Omega^h)} &\le  {P(1/\varepsilon)}h^{{\widetilde \deg} -2}.
\end{aligned}
\end{equation}
We conclude that $\tilde u^{\varepsilon} \in \bar B^h_{\rho}$ provided $h_0e^{ P(1/\varepsilon)}<1$
and 
\beq\label{sufficient2b}
{\widetilde \deg}-2+\widetilde \deg/\mu>\delta.
\eeq
By using the triangle inequality we have
\bea
\|T(w_h)-u^{\varepsilon}\|_{W^{2, \mu}(\Omega^h)} &\le 
\|T(w_h)-T(\tilde u^{\varepsilon})\|_{W^{2, \mu}(\Omega^h)} +\|T(\tilde u^{\varepsilon})-\tilde u^{\varepsilon}\|_{W^{2, \mu}(\Omega^h)} \\
&\quad +\|\tilde u^{\varepsilon}-u^{\varepsilon}\|_{W^{2, \mu}(\Omega^h)}.
\eea 
We consider all three terms separately.

For the first term we use the contraction property \eqref{914}
and deduce that 
\bea \label{916_}
\|T(w_h-\tilde u^{\varepsilon})\|_{W^{2, \mu}(\Omega^h)} &\le
c h^{\eta} \|w_h-\tilde u^{\varepsilon}\|_{W^{2, \mu}(\Omega^h)} \\
&\le c h^{\eta} \|w_h- u^{\varepsilon}\|_{W^{2, \mu}(\Omega^h)} 
+c h^{\eta} \|u^{\varepsilon}-\tilde u^{\varepsilon}\|_{W^{2, \mu}(\Omega^h)} \\
&\le ch^{\eta}\rho + c h^{\eta+{\widetilde \deg}-2+\frac{\widetilde \deg}{\mu}}\|u^{\varepsilon}\|_{C^{\deg}(\bar \Omega^h)}.
\eea

For the second term we set $\xi = u^{\varepsilon}-\tilde u^{\varepsilon}$, $\al (t)= \tilde u^{\varepsilon} + t \xi$ for $0\le t\le1$.
Then, we have in $\Omega^h$
\bea\label{533}
L_{\varepsilon}\left(\tilde u^{\varepsilon}-T(\tilde u^{\varepsilon})\right) &= \Phi_{\varepsilon}(\tilde u^{\varepsilon}) \\
&= \Phi_{\varepsilon}(\tilde u^{\varepsilon}) -\Phi_{\varepsilon}(u^{\varepsilon}) + \Phi_{\varepsilon}(u^{\varepsilon})
\eea
and the right-hand side of this equation is of the form $D_if_i+g\in L^3$ with
\bea
f_i &= -D_{z_i}f_{\varepsilon}(D \tilde u^{\varepsilon}) + D_{z_i}f_{\varepsilon}(D u^{\varepsilon}); \quad 
g = \eta\left(f_{\varepsilon}(D\tilde u^{\epsilon})\right)-\eta\left(f_{\varepsilon}(Du^{\varepsilon})\right).
\eea
Since the finite element space is $H^2$--conforming we may rewrite $D_if_i$
by performing the differentiation and obtain with
\bea
G(t):= D_{z_r}D_{z_i}f_{\varepsilon}(D\alpha(t))D_iD_r\alpha(t)
\eea
that
\bea
D_{z_r}f_i&=
D_{z_r}D_{z_i}f_{\varepsilon}(D\alpha(1))D_iD_r\alpha(1)
-D_{z_r}D_{z_i}f_{\varepsilon}(D\alpha(0))D_iD_r\alpha(0)= \int_0^1\frac{\dd}{\dd t}G(t).
\eea
We have
\bea
\frac{\dd}{\dd t}G(t)&= D_{z_m}D_{z_r}D_{z_i}f_{\varepsilon}(D\alpha(t))
D_m\xi D_iD_r\alpha(t)\\
&\quad + D_{z_r}D_{z_i}f_{\varepsilon}(D\alpha(t))D_iD_r\xi,
\eea
respectively, for a more detailed expansion  see \eqref{expansion_G} now with $w_h$ replaced by $\tilde u^{\varepsilon}$ and $v_h$ by $u^{\varepsilon}$. 
Applying \eqref{estmate_f} and \eqref{estmate_g}  with these replacements, inclusively using the definition of $\xi_h$, we obtain 
\begin{equation}\label{538}
\|D_if_i + g\|_{L^3(\Omega^h)}
\le e^{ P(1/\varepsilon)}
\|u^{\varepsilon}-\tilde u^{\varepsilon}\|_{W^{2,3}(\Omega^h)}
\end{equation}
and hence, using an inverse estimate combined with the stability estimate from Theorem \ref{221} for equation \eqref{533} we obtain
\bea
\|\tilde u^{\varepsilon}-T(\tilde u^{\varepsilon})\|_{W^{2,\mu}(\Omega^h)} &\le
e^{ P(1/\varepsilon)}h^{ \frac{3}{\mu} - \frac{3}{2}} \|D_if_i + g\|_{L^3(\Omega^h)}
\eea
Then, with \eqref{538} and \eqref{901} we get
\bea\label{539}
\|\tilde u^{\varepsilon}-T(\tilde u^{\varepsilon})\|_{W^{2,\mu}(\Omega^h)}&\le e^{ P(1/\varepsilon)}h^{ \frac{3}{\mu} - \frac{3}{2}}\|u^{\varepsilon}-\tilde u^{\varepsilon}\|_{W^{2,3}(\Omega^h)} \\
&\le  e^{ P(1/\varepsilon)}h^{ \frac{3}{\mu} - \frac{3}{2}+{\widetilde \deg} -2+\frac{\widetilde \deg}{3}}\|u^{\varepsilon}\|_{C^{\deg}(\bar \Omega^h)}.
\eea

For the third term we also apply estimate \eqref{901}.

Since the powers of $h$ in \eqref{916_} and \eqref{901} are obviously positive (note that $\tilde{\deg}$ is bounded from below), a sufficient condition for \eqref{915} is given by \eqref{sufficient2b} and 
\begin{align} \label{sufficient2a}
\frac{3}{\mu}-\frac{3}{2}+\widetilde{\deg}-2+\frac{\widetilde \deg}{3} > \delta.
\end{align}
This finishes the proof of part (iii).

We collect sufficient conditions from the  parts (i)--(iii) of the proof, namely, \eqref{suff1a}, \eqref{sufficient2b} and \eqref{sufficient2a} which give (in three space dimensions) the following sufficient condition:

\red{
(a) Case $\mu > 3$:  We have, equivalently to  \eqref{suff1a}, \eqref{sufficient2b} and \eqref{sufficient2a},
\begin{align}\label{345}
\min\left({\widetilde \deg} -2 +\frac{\widetilde \deg}{\mu},\frac{3}{\mu} -\frac{7}{2} + \frac{4}{3} \widetilde{\deg}\right)>\delta>\frac{3}{2}.
\end{align}
Thus, the conditions in \eqref{345} on $\widetilde{deg}$ and $\mu$  can be equivalently rewritten as
\begin{align}\label{1-2}
 {\widetilde \deg} > 7/(2/\mu + 2),\quad \text{and} \quad \widetilde{\deg} > 
 -9/(4\mu) + 15/4.
\end{align}
 For  $\widetilde{\deg}\ge 4$ all $\mu > 3$ satisfy \eqref{1-2}. For $\widetilde{deg} \ge 3$ condition \eqref{1-2} cannot be satisfied.
}

\red{
(b) Case $2\le \mu\le 3$: Here, we have, equivalently to  \eqref{suff1a}, \eqref{sufficient2b} and \eqref{sufficient2a},
\begin{align}\label{545}
\min\left({\widetilde \deg} -2 +\frac{\widetilde \deg}{\mu},\frac{3}{\mu} -\frac{7}{2} + \frac{4}{3} \widetilde{\deg}\right)>\delta> \frac{3}{\mu} + \frac{1}{2}.
\end{align}
Thus, the conditions in \eqref{545} on $\widetilde{deg}$ and $\mu$  can be equivalently rewritten as
\begin{align}\label{1-2-3}
 {\widetilde \deg} > 3/(\mu+1) + 5\mu/(2(\mu+1)),\quad \text{and} \quad \widetilde{\deg} > 
 3.
\end{align}
Discussing the function on the right hand side of the first inequality in the variable $\mu$ we observe that the second condition is the restrictive one; we may choose any $\widetilde{deg}\ge 4$ and  $\mu \in [2,3]$.
}

This  completes the proof of Theorem \ref{main_thm_endversion}.

\if{\section{Conclusion}
This paper improves upper bounds for constants in error estimates for the finite element approximation of regularized elliptic geometric evolution equations. The improved bounds depend polynomially and not exponentially as in \cite{HK} (where explicit bounds at all where derived for the first time) on the inverse regularization parameter. Compared to that reference we now use $H^2$ conforming finite elements (instead of $H^1$ conforming) as main difference which puts us into the setting of the Alexandrov weak maximum principle and which is more realistic from the point of view that these equations originate from modeling the evolution of level sets. This is meant in the sense that $H^2$ conformity of the finite elements 
naturally leads to error bounds on the $H^2$ level which by embedding theorems lead to $C^1$ error estimates and the latter are (almost) required to find level sets. Note that  in general the considered solutions exist only in $C^{0,1}$ but since error estimates arise in form of Sobolev norms of type $W^{m,p}$ with $p< \infty$, the $m\ge 2$ case is of natural interest in order to embed into $C^{0,1}$.
Possible polynomial bounds in the  $H^1$ case are considered for future work.
}\fi

\section*{Acknowledgement}
The second author  has been funded by the Deutsche Forschungsgemeinschaft (DFG,
 German Research Foundation)-Projektnummer: 404870139.
 
 We thank the anonymous referees for their careful proofreading.

 \appendix \label{app1}
 \section{On the realization of Assumptions \ref{hyp1} and \ref{hyp2}}\label{sec:real}

This section is devoted to the explanation that Assumptions  \ref{hyp1} and \ref{hyp2} can be indeed realized. 
We recall first a statement on a polyhedral domain in Lemma~\ref{lem:Zensieck_} using the realization of $H^2$--conforming finite elements from~\cite{MR350260} and general interpolation estimates from~\cite{MR2050138}. In a second step we extend this result to smooth domains as introduced in Section~\ref{sec:main_result}. For this purpose we vary the definition of $\Omega$ within this Appendix \ref{sec:real}.
At first we recall some results from \cite{MR350260} for which we need the following notation:
%
%

Let $ \Omega \subset \mathbb{R}^3$ be a bounded simply or multiply connected domain with boundary $F$ consisting of a finite number of polyhedrons $\Gamma_i$ $(i = 0,\dots, s)$ with $\Gamma_1 ,\dots, \Gamma_s$  lying inside of $\Gamma_0$ and having no intersection. 
Let $\Mcc$ be a set of a finite number of closed tetrahedrons having the following properties: (1) the union of all tetrahedrons is~$\bar \Omega$; (2) two arbitrary tetrahedrons are either disjoint or have a common vertex or a common edge or a common face.
Let $N_t$, $N_v$, and $N_f$ be the total numbers of the tetrahedrons, of the vertices and of the triangular faces in the partition $\Mcc$, respectively. Let the tetrahedrons of $\Mcc$ be denoted by $U_i$ $(i = l,\dots, N_t)$, the vertices by $P_i$ $(i = l,\dots, N_v)$, and the triangular faces by $T_i$ $(i = I,\dots, N_f)$. Let $Q_i$ denote the center of gravity of $T_i$, and $P_{(0)}^i$ the center of gravity of $U_i$.  $Q_{jk}^{(1,s)},\dots, Q_{jk}^{(s,s)}$ denote the points dividing the segment
$\langle P_j,P_k\rangle$ into $s + 1$ equal parts.
The normal to the triangular face whose center of
gravity is $Q_i$ is denoted by $n_i$. We orientate $n_i$ according to the right-hand
screw rule with respect to the increasing indices $j < k < l$ of the vertices $P_j$,
$P_k$, $P_l$ of the face. 

The symbols $s_i$ and $t_i$ denote arbitrary but fixed directions with $n_i$, $s_i$, $t_i$ perpendicular to one another,
 while $s_{jk}$, $t_{jk}$ denote fixed directions such that the directions $P_jP_k$, $s_{jk}$, $t_{jk}$ are perpendicular to one another.

\if{
For the realization of Assumptions  \ref{hyp1} and \ref{hyp2} we recall results from \cite{MR350260} where $H^2$--conforming polynomial approximations on tetrahedrons in three dimensions are considered and error estimates are derived. However, we remark that in that reference the considered domain is polygonal and not smooth as in our setting.
}\fi

We set $D^{\alpha}f:=\partial^{|\alpha|_1}/\partial^{\alpha_1} \partial^{\alpha_2} \partial^{\alpha_3}f$, and for $\beta=(\beta_1,\beta_2)$ and $|\beta|_1=\beta_1+\beta_2$ we define $D^{\beta}_{jk} f:=\partial^{|\beta|_1}/\partial s_{jk}^{\beta_1}\partial t_{jk}^{\beta_2}f$. We denote by  
 $\partial f/\partial s_i$, $\partial f/\partial t_i$, $\partial f/\partial s_{jk}$, and $\partial f/\partial t_{jk}$  the derivatives of the function $f$ in the directions $s_i$, $t_i$, $s_{jk}$, and $t_{jk}$.
 \if{
  Let there be prescribed at each point $P_i$ thirty-five values $D^{\alpha}f(P_i)$ $(| \alpha|_1\le 4)$, at each point $Q_{jk}^{(1,1)}$ two values $D_{jk}^{\beta}f(Q_{jk}^{1,1})$ $(|\beta|_1 = 1)$, at each point $Q_{jk}^{(r,2)}$ three values $D_{jk}^{\beta} f(Q_{jk}^{r,2})$ $(|\beta|_1=2$), at each point $P_0^{(i)}$ four values $D^{\alpha} f(P_0^{(i)})$ $( |\alpha|_1 \le  1)$ and at each point $Q_i$ one value $f(Q_i)$ and six values $D_i^{\beta} \partial f(Q_i)/\partial n_i$ $(|\beta|_1\le 2)$. Then on each tetrahedron $D_i$ this determines a unique polynomial of ninth degree $p_i(x, y, z)$.}\fi
  We introduce the following set of degrees of freedom characterizing a function $w\in C^1(\bar \Omega)$ being on each tetrahedron $D_i$ a  polynomial of degree $9$, see \cite[Thm. 2]{MR350260}:
\begin{equation}\label{conditions-poly}
 \left\{
 \begin{aligned}
&D^{\alpha}w(P_i) , \quad |\alpha|_1 \le 4; \\
&D^{\beta}_{jk} w(Q_{jk}^{(r,s)}) ,\quad |\beta|_1=s,\quad  r=1,\dots,s,\quad s=1,2; \\
& w(Q_i) ; \\
& D_i^{\beta}\left( \frac{\partial w(Q_i)}{\partial n_i}\right) ,\quad |\beta|_1\le 2; \\
& D^{\alpha}w(P_0^{(i)}) , |\alpha|_1 \le 1 
\end{aligned}
\right.
\end{equation}
where $i = 1,\dots, 4$, $j = 1,2, 3$, $k = 2,3,4$ $(j < k)$.
\if{By Corollary \cite[Cor. 1]{MR350260} a polynomial of the ninth degree is uniquely determined by the conditions 
\begin{equation}\label{conditions-poly}
 \left\{
 \begin{aligned}
&D^{\alpha}w(P_i) = 0, \quad |\alpha|_1 \le 4; \\
&D^{\beta}_{jk} w(Q_{jk}^{(r,s)}) = 0,\quad |\beta|_1=s,\quad  r=1,\dots,s,\quad s=1,2; \\
& w(Q_i) = 0; \\
& D_i^{\beta}\left( \frac{\partial w(Q_i)}{\partial n_i}\right) = 0,\quad |\beta|_1\le 2; \\
& D^{\alpha}{\alpha}w(P_0) = 0, |\alpha|_1 \le 1 
\end{aligned}
\right.
\end{equation}
with  $i = 1,\dots, 4$, $j = 1,2, 3$, $k = 2, 3,4$ $(j < k)$ and  $w=p-f$, $f$ a four-times continuously differentiable function on the tetrahedron $\bar{U}$. }\fi
Denoting  by $W_h$ the set of these functions, it is of dimension
\bea
\dim W_h=35N_v, + 7N_f, + 4N_t, + 8N_e
\eea
with $N_v$ the number of vertices, $N_f$ the number of triangular faces, $N_t$ the number of tedrahedrons, and $N_e$ the number of edges, and $h$ the mesh parameter.

We further remark that the simplest polynomial $p(x,y,z)$ on a tetrahedron which leads to piecewise polynomial functions which are globally continuously differentiable
is expected to be of order nine~\cite{MR350260}.

\begin{lemma}\label{lem:Zensieck_} (i) For $h>0$ the ansatz space 
\beq
\begin{aligned} \label{def_V_h}
V_h:=& \set*{
 w \in C^1(\bar \Omega) \; \given 
\begin{array}{l}\text{for all } T\in \mathcal{T}_h:\;\; w|_T \text{ polynomial of degree } 9\\ \text{characterized by \eqref{conditions-poly}}
 ,\quad 
  w_{|\partial \Omega^h}=0
  \end{array}
   }
\end{aligned}
\eeq
is $H^2$--conform.  

(ii) For given  $u \in W^{\deg,\infty}( \Omega )$ with $\deg$ as in Assumption \ref{hyp1}  and $0<h\le h_0$ there exists a function $\hat{\mathcal{I}}_hu:=\varphi\in V_h$ 
 having the same values at the points $P_i$,
$Q_{jk}^{(r,s)}$, $P_{0}^{(i)}$, $Q_i$ as $u$, such that for~$1\le p \le \infty$ that
\begin{equation}\label{866-2}
\begin{aligned}
\|u-\hat{\mathcal{I}}_hu\|_{W^{m,\infty}(\Omega)} \le ch^{\deg-m}\|u\|_{W^{{\deg,\infty}}( \Omega)}, \quad m =1, 2.
\end{aligned}
\end{equation}
\end{lemma}

\begin{proof}
 (i) By the consideration above, cf. \cite{MR350260}, the space $V_h$ is well-defined, the conformity follows by a classical result, cf. \cite[Thm. 5.2]{MR2322235}.
 
 (ii) Consequence of \cite[Thm. 1.109]{MR2050138} setting in this reference $l$ equal to the $m$ from above, and $p=\infty$. 
 \if{XXX REPRODUCE THE PROOF OF INTERPOLATION ESTIMATES  INDEPENDENT OF ZENSIEK By \eqref{inverse_estimate_lemma} we have
 We consider the ansatz space as polynomials of degree 
 
 \beq
 |D^{\alpha} u|\le C h^{\deg-|\alpha|_1},
 \eeq 
 in  \cite[Proof of Thm. 4]{MR350260}, since
 \beq
 \|u-\hat{\mathcal{I}}_hu\|_{H^{2}(\Omega)} \le C |D^{\alpha} u| h^8,\quad |\alpha|=10
 \eeq
 }\fi
\end{proof}

\if{
The author shows under certain hypotheses an estimate of type
\begin{align}
 \norm{u- \varphi}{H^2(\Omega)}\leq C h^8
\end{align}
cf. , between a given function $u$ having bounded derivatives of tenth order and an approximating function $\varphi$ in the ansatz space.
}\fi
In the following we explain how Assumptions \ref{hyp1} and \ref{hyp2}
can be realized by using Lemma \ref{lem:Zensieck_}. On the one hand this can be found in standard text books (e.g. in the $H^1$ case and for convex and smooth domains) as `curved boundary elements' or `boundary approximation', on the other hand we do not have a concrete reference for our specific scenario, 
so we present the following argument for convenience. 
Let $\Omega$ be again as in Section \ref{sec:main_result} and $\hat \Omega^h$ a triangulation where we assume that all faces of the boundary elements are flat (and not curved). 
Let $\hat V_h$ be the finite element space according to Lemma \ref{lem:Zensieck_} on the polygonal domain $\hat \Omega^h$. We will now construct $V_h$ and $\Omega^h$ which satisfy Assumptions~\ref{hyp1} and~\ref{hyp2} on the basis of the previously mentioned spaces.
As already explained $d$ denotes the signed distance function with respect to $\partial \Omega$ such that
\beq
d_{|\Omega} <0, \quad d_{|(\mathbb{R}^3\setminus\bar \Omega)} >0, \quad d_{|\partial \Omega}=0.
\eeq 
Let $\tau>0$ and consider
\beq
\Omega_{\tau}=\{d<\tau\}, \quad \Omega_{-\tau}= \{d<-\tau\},
\quad U_{\tau}=\Omega_{\tau}\setminus \overline {\Omega_{-\tau}}=\{|d|<\tau\}.
\eeq
Note that $\tau$ is here an auxiliary variable which is not the regularization parameter from the previous sections. 
Let $N_h$ be the set of nodes of the triangulation of $\hat \Omega^h$. We may assume $h=o(\tau)$ so that 
\beq
\partial \hat \Omega^h \subset U_{\tau}.
\eeq
We would like to define a diffeomorphism
\beq \label{phi1}
\Phi:\Omega_{\tau}\rightarrow \check \Omega, \quad \Phi_{|\Omega_{-2\tau}}=\id,
\quad \Phi(N_h \cap \partial \hat \Omega^h)\subset \partial \Omega
\eeq
where $\check \Omega$ is an open auxiliary set containing $\Omega_{-\tau}$ and 
so that
\beq \label{phi2}
|D\Phi|+|D^k\Phi|\le c, \quad k=2,3,4,
\eeq
uniformly in $h$ and roughly spoken so that the triangulation using the nodes $N_h$ carries over via $\Phi$ to a triangulation of $\Phi(N_h)$, meaning that the distortion of the nodes is not too large. Without making this formulation precise here we will see that we are far away from such a kind of criticality.

We can define
\beq
\Omega^h=\Phi(\hat \Omega^h),
\eeq
\beq
V_h=\{\varphi\circ \Phi^{-1}: \varphi \in \hat V_h\}
\eeq
and
\beq
{\mathcal{I}}_hu = \hat{\mathcal{I}}_h(u\circ \Phi)\circ \Phi^{-1}
\eeq
for $u \in H^{\deg}(\Omega^h)$.
Introducing coordinates $p=(\hat x, x_{3})\in U_{\tau}$, $\hat x\in \partial \Omega$, $\dist(\hat x, p)=x_{3}$, we can write $N_h \cap \partial  \hat  \Omega^h$ as a graph over a suitable
discrete set $D \subset \partial \Omega$, i.e. there is $u:D \rightarrow \mathbb{R}$ such that
\beq
N_h \cap \partial \hat \Omega^h =\{(\hat x, u(\hat x)), \hat x \in D\}.
\eeq
Clearly, we can extend $u$ to $\partial \Omega$ as a smooth function such that
\beq \label{alles_klein3}
\frac{|u|}{h^2}+\frac{|Du|}{h}+|D^ku|\le c, \quad k=2,3,4,
\eeq
where $|\cdot|$ refers to the Euclidean norm of $Du$ and $D^ku$, $k=2,3,4$, with respect to a fixed selection of finitely many local coordinate systems. (Note that one has here the choice of more specific extensions, e.g. piecewise polynomial and sufficiently regular, depending on the concrete situation.)
Let $\rho \in C^{\infty}(\mathbb{R})$ 
such that
\beq \label{alles_klein}
\rho(t) =
\begin{cases}
0, \quad t \le -2 \tau \\
1, \quad t \ge -\tau
\end{cases}
\eeq
and
\beq \label{alles_klein2}
0 \le \rho, \quad 0 \le \rho' \le \frac{c}{\tau}, \quad |\rho''|\le \frac{c}{\tau^2}.
\eeq
Note that on the $C^4$ level instead of the $C^{\infty}$ case anything like this can be achieved by using polynomial functions as well. The specifical choice of $C^4$ is related to the desired order of the boundary approximation.
Then $\Phi$ is now defined as follows
\beq
\Phi(y)=y, \quad y \in \Omega_{-2\tau}
\eeq
and
\beq
\Phi: U_{2\tau} \cap \Omega_{\tau} \ni (\hat x, x_{3}) \mapsto (\hat x, x_{3}-u(\hat x)\rho(x_{3})) \in \mathbb{R}^3\setminus \Omega_{-2\tau}.
\eeq
A calculation shows that $\Phi$ satisfies \eqref{phi1} and \eqref{phi2} provided $h=h(\tau)$, e.g., $h \le \tau^2$, is sufficiently small: The Jacobian of $\Phi$
is given by
\beq
\left(
\begin{array}{rr}                                
I & 0  \\                                               
-\rho(x_{3})Du(\hat x) & 1-u(\hat x)\rho'(x_{3}) \\                                                                                              
\end{array}
\right)
\eeq
and in view of \eqref{alles_klein}, \eqref{alles_klein2} and \eqref{alles_klein3} invertible. And concerning injectivity of $\Phi$ we have that
\beq
(\hat x, x_{3}-u(\hat x)\rho(x_{3})) = (\hat y, y_{3}-u(\hat y)\rho(y_{3}))
\eeq
implies $\hat x=\hat y$ and in consequence that
\beq
x_{3}-y_{3}= u(\hat x)(\rho(x_{3})-\rho(y_{3}))=
u(\hat x)\rho'(\xi) (x_{3}-y_{3})
\eeq
with some $\xi$
which implies $x_{3}=y_{3}$ since $u(\hat x)\rho'(\xi)$ is small. Note that the so piecewisely defined $\Phi$
defines a diffeomorphism from $\Omega_{\tau}$ onto a set $\check \Omega$ as desired.

Note that according to \eqref{conditions-poly} there are 13 points on each face of a tetrahedron where values (of zero order or higher order derivatives) are prescribed. Note that among these 13 points are 10 points given by the corners of the (triangular) face, the mid point of the face, and two  points on each of the three edges of the face which
divide each edge into 3 pieces of equal length. Furthermore,  prescribing only function values, these 10 points define a unique interpolation operator over the face with polynomials in two variables of degree at most three.
We remark that if one replaces the two interpolation points on each edge by the Gauss points one obtains a well-known approach for conforming piecewise cubic finite elements in two variables on triangles. Obviously, this also applies to our case with two points separating each edge in three parts of the same length.

Our construction above moves for each `boundary face' these 13 points into the boundary $\partial \Omega$. Since this is achieved with a uniform bound for the  fourth order derivatives of the fitting surface and $\partial \Omega$ (which is compact),  we conclude that the fitting surface from our above construction and $\partial \Omega$ solve the interpolation problem over each boundary face (before the move) specified by prescribing function values in each of these 10 points. Hence the distance of the fitting surface and $\partial \Omega$  from an 
interpolating surface (used here as an auxiliary tool) of the aforementioned interpolation problem is of order~$h^4$ and hence by triangle inequality also the distance between the fitting surface and $\partial \Omega$. This shows that the required boundary approximation from Assumption~\ref{hyp2} is achieved with $\widetilde{\deg}=4$.

\begin{remark}\label{app:inv}
The gap in the proof of Lemma \ref{inverse_estimate_lemma} given the reference \cite[Sec. 4.5]{MR1278258} can straightforwardly be filled by the consideration above. Also the interpolation estimate \eqref{866} follows straightforwardly from the one above. As an example we explain the latter in detail. Let $u \in H^{\deg}(\Omega^h)$. In order to show that $\mathcal{I}_h u -u$ satisfies \eqref{866}, we apply \eqref{866-2} (i.e. here with $ \Omega=\hat \Omega^h$). For it  we first rewrite
\begin{align} \label{23445}
\mathcal{I}_h u- u = \hat{\mathcal{I}}_h (u \circ \Phi) \circ \Phi^{-1}-u
\end{align}
with $\hat{\mathcal{I}}_h$ defined in Lemma \ref{lem:Zensieck_} 
Since $\Phi$ has bounded derivatives up to order 2, we may instead of the previous quantity equivalently estimate 
$(\mathcal{I}_hu-u)\circ \Phi= \hat{\mathcal{I}}_h (u \circ \Phi)-u \circ \Phi$. The last expression satisfies the available property \eqref{866-2}. Hence  \eqref{23445} can be estimated in the desired Sobolev norm by the corresponding Sobolev norm of $u \circ \Phi$. The latter can be estimated from above by the same norm of $u$ up to a multiplicative constant depending on $\Phi$ (which is hence uniform in $h$).

\end{remark}

\section{$L^{\infty}$--estimate with exponential asymptotics in the constant with respect to the bounds of the coefficients}\label{sec:est}
\numberwithin{equation}{section}

In this section we derive an $L^{\infty}$--estimate for strong solutions for linear equations
which is obtained from redoing the proof of the Alexandrov weak maximum principle for strong solutions from~\cite[Thm.~9.1]{MR737190}. In that reference the supremum is estimated by a product
consisting of a  first factor being a constant which depends on the differential operator and the domain 
and a second factor being the $L^{n+1}$ norm of the right-hand side (assuming an $(n+1)$-dimensional domain). For our purposes we need 
an explicit dependence of the constant in this product on the assumed bounds for the coefficients of the differential operator.  
The formulation in \cite[Thm.~9.1]{MR737190} itself does not provide such an explicit dependence. When following the proof of the latter  reference one can straightforward extract an explicit relation in addition to the statement of the theorem. According to this explicit relation the constant grows exponentially
in the bounds for the coefficients. We recall that proof for convenience in detail and adopt for this section the notation from \cite{MR737190}. 
We consider the differential operator
\begin{equation}
Lu := a_{ij}D_iD_j u + b^iD_i u + cu
\end{equation}
with measurable, bounded coefficients $a_{ij}$, $b^i$ and $\tilde c$ on $\Omega$
satisfying the bounds
\begin{equation}
0<\lambda \le a_{ij} \le \Lambda\ (\text{in the sense of quadratic forms}), \quad |b^i| \le c_1, \quad |\tilde c| \le c_2, \quad \tilde c \le 0
\end{equation}
in $\Omega$ where $\lambda$, $\Lambda$, $c_1$ and $c_2$ are some positive constants.

\begin{theorem} (\cite[Thm.  9.1]{MR737190})\label{new_sup_estimate}
Let $f \in L^{n+1}(\Omega)$, $\tilde c\le 0$ and $u \in C^0_0(\bar \Omega)\cap W^{2, n+1}_{loc}(\Omega)$ with 
\begin{equation}
Lu\ge f.
\end{equation} 
Then
\begin{equation}\label{new_sup_estimate-est}
\sup_{\Omega} u \le e^{P(\lambda, \Lambda, c_1, c_2)} (\diam \Omega) \|f\|_{L^{n+1}(\Omega)};
\end{equation}
recall the convention for $P$ given in \eqref{explain_P}.
\end{theorem}

\begin{proof}
 Note that the constant $P(\lambda, \Lambda, c_1, c_2)$ in Theorem \ref{new_sup_estimate} does not depend on $\Omega$ and that the functions $\diam(\cdot)$ and $\|f\|_{L^{n+1}(\cdot)}$ are monotone with respect to the partial ordering of sets by inclusion, hence we may---by considering connected components of $\{x\in \Omega: u(x)>0\}$---assume for the proof that $u\ge 0$ in $\Omega$.
 We use the following lemma, cf. \cite[Lemma.  9.4]{MR737190}.
\begin{lemma}
\label{help_lemma}
Let $g$ be a nonnegative, locally integrable function on $\mathbb{R}^{n+1}$. Then, for any $u \in C^2(\Omega) \cap C^0(\bar{\Omega})$, we have
\begin{align}
&\int_{B_{\tilde{M}}(0)} g  \le \frac{1}{((n+1)\lambda)^{n+1}}\int_{\Gamma^{+}} g(Du) \left(- a_{ij}D_iD_ju \right)^{n+1}
\end{align}
where 
\begin{equation}
\tilde{M} := (\sup_{\Omega} u  - \sup_{\partial \Omega} u)/d,\quad d:= \operatorname{diam} \Omega
\end{equation}
and 
\begin{equation}
\Gamma^{+} := 
\left\{
y\in \Omega \; \bigg| \;
\begin{array}{l}
u(x) \le u(y)+p\cdot (x-y) \quad \text{for all } x \in \Omega\\
\text{ and some } p=p(y)\in \mathbb{R}^{n+1} 
 \end{array}
 \right\}.
\end{equation}
\end{lemma}
Note that $D^2u$ is nonpositive in $\Gamma^{+}$.
First, we show Theorem \ref{new_sup_estimate} under the assumption that $u \in C^2(\Omega)\cap C^0(\bar \Omega)$ and second, we deduce the general case
$u \in C^0_0(\bar \Omega)\cap W^{2, n+1}_{loc}(\Omega)$ by an approximation argument. So let us assume this higher regularity for $u$. We will use Lemma~\ref{help_lemma} with
\begin{equation}
g(p) := \left(|p|^{\frac{n+1}{n}}+\mu^{\frac{n+1}{n}}\right)^{-n}, \quad p \in \mathbb{R}^{n+1},
\end{equation}
where  $\mu>0$ is a parameter which will be set later on to $\|f\|_{L^{n+1}(\Omega)}$ provided this norm does not vanish.
Now we estimate in $\Gamma^{+}$
\begin{equation}
 \begin{aligned}
g(Du) \left(- a_{ij}D_iD_ju \right)^{n+1} &\le
\frac{(|b||Du|+|f|)^{n+1}}{\left(|Du|^{\frac{n+1}{n}}+\mu^{\frac{n+1}{n}}\right)^{n}}\\
&\le 
c\frac{|b|^{n+1}|Du|^{n+1}+|f|^{n+1}}{\left(|Du|^{\frac{n+1}{n}}+\mu^{\frac{n+1}{n}}\right)^{n}}\\
&\le 
c
\left(|b|^{n+1}+\frac{|f|^{n+1}}{\mu^{n+1}}\right).
\end{aligned}
\end{equation}
From below we estimate
\begin{equation}
g \ge 2^{(1-n)}\left(|p|^{n+1}+\mu^{n+1}\right)^{-1}.
\end{equation}
Integration of the lower bound yields
\begin{equation}\begin{aligned}
\int_{B_{\tilde{M}}(0)} \left(|p|^{n+1}+\mu^{n+1}\right)^{-1} &=
c \int_0^{\tilde M} \frac{r^{n}}{\left(r^{n+1}+\mu^{n+1}\right)}\dd r\\
&= c \log\left(r^{n+1}+\mu^{n+1}\right)\bigg|_0^{\tilde M} \\
&=c \log \left(\tilde M^{n+1}+\mu^{n+1}\right) -c \log (\mu^{n+1}).
\end{aligned}\end{equation}
Together we obtain 
\begin{equation}\begin{aligned}
\tilde M &\le
\left[\exp  \left(P(\lambda, \Lambda, c_1, c_2)
\int_{\Omega} \left(|b|^{n+1}+\frac{|f|^{n+1}}{\mu^{n+1}}\right)\dd x \right)\mu^{n+1}\right]^{\frac{1}{n+1}} \\
\end{aligned}\end{equation}
where we assumed for the last two equations that
 $\|f\|_{L^{n+1}(\Omega)}\neq 0$ and set $\mu = \|f\|_{L^{n+1}(\Omega)}$. The $f=0$ case is clear.
The general case concerning the assumed regularity for $u$ follows by using an approximation argument as in the proof of \cite[Lem.~9.4]{MR737190}, being
 even simpler as in the latter reference in view of the available uniform ellipticity of $L$
in $\Omega$.
\end{proof}

\section{Sobolev embeddings}\label{app:C}
\numberwithin{equation}{section}

We recall from  \cite{MR3497775} the continuous embedding of Sobolev in H\"older spaces.

Let $\Omega \subset \mathbb{R}^{n+1}$ be open and bounded with Lipschitz boundary,
$m \ge 1$ and $k \ge 0$ be integers, and $1 \le p < \infty$. Then the following holds:
 If
 \begin{equation}\label{sob-2}
m - \frac{n+1}{p} \ge k + \alpha
\end{equation}
and
 $0 < \alpha < 1$,
then the embedding
\begin{align}
\id \colon W^{m,p} (\Omega ) \subset C^{k,\alpha} (\bar \Omega)
\end{align}
exists and is continuous.

\providecommand{\bysame}{\leavevmode\hbox to3em{\hrulefill}\thinspace}
\providecommand{\MR}{\relax\ifhmode\unskip\space\fi MR }
\providecommand{\MRhref}[2]{%
  \href{http://www.ams.org/mathscinet-getitem?mr=#1}{#2}
}
\providecommand{\href}[2]{#2}

\end{document}